\PassOptionsToPackage{shortlabels}{enumitem}
\PassOptionsToPackage{table}{xcolor} 
\documentclass[11pt]{article}

\usepackage[compress]{natbib}
\usepackage{amsmath,amssymb,amsthm}
\usepackage{color} 
\usepackage[margin=1in]{geometry}
\usepackage{setspace}
\usepackage{amsthm, amsmath, amssymb, amsfonts}
\usepackage{bm}
\usepackage[dvipsnames,table,xcdraw]{xcolor}

\usepackage{lipsum}
\usepackage{appendix}
\usepackage{amsfonts}
\usepackage{graphicx}
\usepackage{epstopdf}
 \usepackage[ruled,vlined]{algorithm2e}   
\definecolor{colorSEG}{HTML}{FDEBD0}  
\definecolor{colorSOGDA}{HTML}{E1F5FE}
\usepackage{enumitem}   
\ifpdf
  \DeclareGraphicsExtensions{.eps,.pdf,.png,.jpg}
\else
  \DeclareGraphicsExtensions{.eps}
\fi

\allowdisplaybreaks[4]

\usepackage{amsopn}
\usepackage{mathtools}   
\usepackage{url}            
\usepackage{booktabs}       

\usepackage{mathrsfs} 
\usepackage{hyperref} 
\hypersetup{
    colorlinks=true,
    breaklinks=true,
    linkcolor=red!80!black,
    citecolor=blue!70!black,
    urlcolor=blue!70!black
}
\usepackage{cleveref}
\usepackage{multirow}

\usepackage{tabularx}
\usepackage{colortbl}
\usepackage{textcomp}

\usepackage{caption}
\usepackage{subcaption}
 
\usepackage{tcolorbox} 
\usepackage{xr}

\usepackage{caption}
\usepackage{url}            
\usepackage{booktabs}       
\usepackage{nicefrac}       
\usepackage{graphicx}
\usepackage{caption}
\usepackage{thm-restate}
\usepackage{bbm}
\usepackage[export]{adjustbox}

\usepackage{microtype}
\usepackage{mathpazo} 
\usepackage[scaled]{helvet} 
\usepackage{courier} 
\normalfont
\usepackage[T1]{fontenc}

\usepackage{tikz}
\usetikzlibrary{shadows}
\usetikzlibrary{arrows,positioning} 
\def\arraystretch{1.5}
\usepackage{tcolorbox}

\def\O#1{\text{\ding{\the\numexpr#1+171}}}

\definecolor{amethyst}{rgb}{0.6, 0.4, 0.8}
\newcount\Comments
\newcommand{\kibitz}[2]{\ifnum\Comments=1{\textcolor{#1}{\textsf{\footnotesize #2}}}\fi}

\usepackage{amsmath}
\usepackage{graphicx} 
\usepackage{latexsym}
\usepackage{mathrsfs}
\usepackage{mathtools}
\usepackage{appendix}
\usepackage{multirow}
\usepackage{makecell}
\usepackage{mathtools}
 \usepackage{amssymb}
\usepackage[misc]{ifsym}
\usepackage{bm}
\usepackage{harmony}
\usepackage{scalerel}
\usepackage{booktabs}
 \usepackage{color}
\usepackage{verbatim}
\usepackage{url}
\usepackage{hyperref}
\usepackage{graphicx,graphics}
\usepackage{multirow}
\usepackage[shortlabels]{enumitem} 
\usepackage{pgfplots}
\usepackage{threeparttable}
\allowdisplaybreaks[4] 
\usepackage[english]{babel} 
\usepackage{blindtext}
\usepackage{float}
\usepackage[export]{adjustbox}
 
\newcounter{myalg}
\renewcommand{\themyalg}{\arabic{myalg}}
\makeatletter
\newcommand{\setmyalglabel}[1]{%
  \protected@edef\@currentlabel{\themyalg}%
  \label{#1}%
}
\makeatother 
\newcounter{spb}
\def\bu{\bm{u}}

\def\x{\bm{x}}
\def\y{\bm{y}}
\def\z{\bm{z}}
\def\v{\bm{v}}
\def\u{\bm{u}}
\def\e{\bm{e}}

\def\p{\bm{p}}

\def\b1{\bm{1}}

\newcommand{\F}{\mathcal F}
\newcommand{\R}{\mathbb R}
\newcommand{\E}{\mathbb E}

\newcommand{\X}{\mathcal X}
\newcommand{\Y}{\mathcal Y}

\newcommand{\mO}{\mathcal O}

\newcommand{\U}{\mathcal U}

\DeclareMathOperator{\proj}{proj}

\DeclareMathOperator{\diam}{diam}
\DeclareMathOperator{\dist}{dist}

\DeclareMathOperator*{\argmin}{argmin}
\DeclareMathOperator*{\argmax}{argmax}

\usepackage{thmtools}
 
\declaretheoremstyle[parent=section]{definitionwithend}
\declaretheorem[style=definitionwithend]{corollary}
\declaretheorem[style=definitionwithend]{theorem}

\declaretheorem[style=definitionwithend]{definition}
\declaretheorem[style=definitionwithend]{assumption}

\declaretheorem[style=definitionwithend]{remark}

\declaretheorem[style=definitionwithend]{lemma}

\usepackage{appendix}
\usepackage{pifont}
\title{Last-Iterate Convergence of Single-Loop Stochastic Methods for Constrained Convex-Concave Minimax Problems} 

\date{\today}

\author{%
    Taoli Zheng\thanks{RIKEN AIP  \texttt{taoli.zheng@riken.jp}} \and
    Jiajin Li\thanks{Sauder School of Business, University of British Columbia, Vancouver, BC, Canada. \texttt{jiajin.li@sauder.ubc.ca}} \and
    Anthony Man-Cho So\thanks{Department of Systems Engineering and Engineering Management, The Chinese University of Hong Kong, Shatin, NT, Hong Kong. \texttt{manchoso@se.cuhk.edu.hk}} 
}
\begin{document}

\maketitle
 
\begin{abstract}
In this paper, we study last-iterate convergence of stochastic first-order methods for constrained smooth convex--concave minimax optimization under the standard bounded-variance stochastic oracle. A fundamental challenge is that the last iterates of vanilla stochastic extragradient (S-EG) and stochastic optimistic gradient descent--ascent (S-OGDA) may fail to converge in the presence of stochastic gradient noise, even for simple bilinear problems. To overcome this difficulty, we introduce a simple perturbation framework that regularizes the original convex--concave problem into a strongly convex--strongly concave one. Applying S-EG and S-OGDA to the perturbed problem yields two simple single-loop methods, referred to as perturbed S-EG (PS-EG) and perturbed S-OGDA (PS-OGDA).

We establish last-iterate convergence by first deriving convergence in terms of the squared distance to the saddle point of the perturbed problem and then translating this estimate into guarantees for the restricted primal--dual gap. Based on this framework, we establish two types of convergence guarantees. When the optimization horizon is known \emph{a priori}, both PS-EG and PS-OGDA achieve an $\mathcal{O}(T^{-1/4})$ last-iterate convergence rate for the restricted primal--dual gap, which coincides with the standard primal--dual gap on compact feasible domains. When the optimization horizon is unknown, we develop an anytime variant based on diminishing perturbations and diminishing stepsizes. For general closed convex feasible sets, both PS-EG and PS-OGDA achieve an $\mathcal{O}(T^{-1/5})$ last-iterate convergence rate for the restricted primal--dual gap. Furthermore, in the unconstrained setting, PS-EG admits a sharper $\mathcal{O}(T^{-1/4})$ anytime convergence rate in terms of the gradient norm.
\end{abstract}
 
\section{Introduction}\label{sec:intro}

In this paper, we consider the stochastic convex--concave minimax problem
\begin{equation}
\min_{\x\in\X}\max_{\y\in\Y}
F(\x,\y)
\coloneqq 
\mathbb{E}_{\bm\zeta}[f(\x,\y;\bm\zeta)],
\tag{P}
\label{eq:prob}
\end{equation}
where $f(\cdot,\cdot;\bm\zeta):\mathbb{R}^n\times\mathbb{R}^d\rightarrow\mathbb{R}$ is smooth and convex--concave for almost every realization of $\bm\zeta$, and $\X\subseteq\mathbb{R}^n$, $\Y\subseteq\mathbb{R}^d$ are closed convex sets. We assume access only to unbiased stochastic first-order oracles with bounded variance.

While the ergodic convergence of stochastic first-order methods is by now well understood, understanding their \emph{last-iterate} behavior remains considerably more challenging. This distinction is practically important since optimization algorithms  always return the final iterate rather than an average of the trajectory. Consequently, a central question in stochastic minimax optimization is whether simple first-order methods can achieve non-asymptotic last-iterate convergence guarantees under the standard stochastic oracle model.

Among first-order methods, extragradient (EG)~\citep{korpelevich1976extragradient} and optimistic gradient descent--ascent (OGDA)~\citep{popov1980modification} are arguably the two most widely used algorithms for solving convex--concave minimax problems. In deterministic smooth settings, their last-iterate behavior is now well understood: Both algorithms achieve optimal convergence rates in terms of the primal--dual gap, matching known lower bounds~\citep{nemirovski2004prox,mokhtari2020convergence,golowich2020last,cai2022tight,xie2020lower}. In stochastic settings, however, the situation is
fundamentally different. Although averaged iterates achieve the classical $\mathcal{O}(T^{-1/2})$ ergodic convergence rate under unbiased stochastic first-order oracles with bounded variance~\citep{juditsky2011solving,mishchenko2020revisiting,gidelvariational}, these guarantees do not imply convergence of the last iterate. The reason is that the stochastic noise may destabilize the dynamics. Even for simple bilinear games, the last iterates generated by vanilla stochastic EG (S-EG) and stochastic OGDA (S-OGDA) may fail to converge~\citep{chavdarova2019reducing,ryu2019ode,hsieh2020explore}; see Figure~\ref{fig:trajectory} for details. Therefore, obtaining last-iterate convergence requires additional mechanisms to stabilize the stochastic dynamics.

\begin{figure}[ht]
    \centering
    \begin{minipage}{0.45\textwidth}
        \centering
        \includegraphics[width=\linewidth]{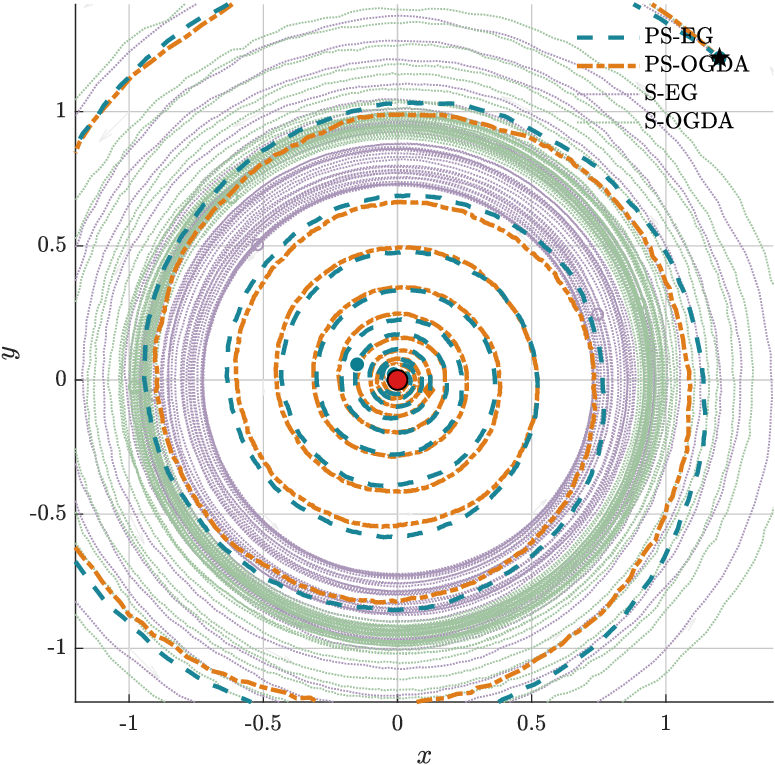}
        \subcaption{Comparison of vanilla and perturbed S-EG/S-OGDA on a stochastic bilinear problem. Vanilla methods fail to converge, whereas the perturbed variants converge.}
        \label{fig:trajectory}
    \end{minipage}
    \hfill
    \begin{minipage}{0.45\textwidth}
        \centering
        \includegraphics[width=\linewidth]{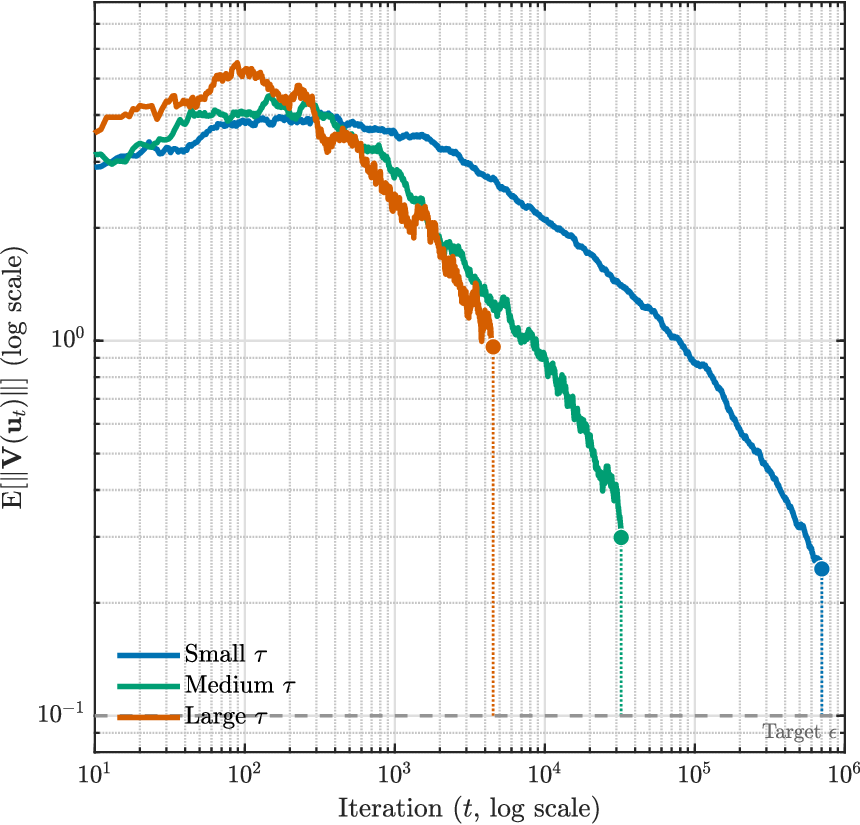}
        \subcaption{Influence of the perturbation parameter $\tau$. Larger $\tau$ accelerates convergence but introduces a larger asymptotic bias.}
        \label{fig:convergence}
    \end{minipage}
    \vspace{-2mm}
\end{figure}


Existing approaches can be broadly grouped into two families. The first family reduces the effect of stochastic noise at the oracle or estimator level, often by using additional samples or stronger sampling structure. Representative examples include mini-batch methods, which query multiple stochastic gradients per iteration, and variance-reduced methods, which construct low-variance gradient estimators by exploiting finite-sum structure or recursive stochastic estimation~\citep{lee2021fast,cai2022stochastic}.

The second family stabilizes the stochastic dynamics at the algorithmic level while retaining the standard stochastic oracle. Representative examples include anchoring techniques and adaptive perturbation schemes. For constrained stochastic minimax problems with compact sets, ~\cite{abe2024adaptively} proposed adaptively perturbed mirror descent and established a $\widetilde{\mathcal{O}}(T^{-1/10})$ last-iterate convergence rate. Subsequent anchoring-based methods improved this rate to $\widetilde{\mathcal{O}}(T^{-1/7})$~\citep{abe2025boosting}. Under stronger assumptions, including Hessian Lipschitz continuity, generalized Frank--Wolfe methods further achieve an $\mathcal{O}(T^{-1/6})$ rate~\citep{chen2024last}. Nevertheless, under the standard bounded-variance oracle and smoothness assumptions, whether one can obtain faster last-iterate convergence using simple stochastic EG/OGDA-type methods remains open.


In this paper, we adopt a perturbation-based stabilization strategy. Specifically, we add a small quadratic perturbation that renders the objective strongly convex in the primal variable and strongly concave in the dual variable, and then directly apply the standard stochastic EG or stochastic OGDA method to the perturbed problem. The resulting perturbed stochastic EG (PS-EG) and perturbed stochastic OGDA (PS-OGDA) preserve the simplicity of the original algorithms, require no additional oracle model beyond vanilla S-EG/S-OGDA, and introduce neither variance reduction, nor increasing batch sizes.

The central challenge is to determine an appropriate perturbation level. A large perturbation induces strong contraction but also introduces significant approximation bias, whereas a small perturbation better approximates the original problem but provides insufficient stabilization against stochastic noise; see Figure~\ref{fig:convergence}. This trade-off motivates the development of our perturbation framework, whose main contributions are summarized below.
  
\begin{itemize}
\item We introduce a perturbation-based framework that stabilizes stochastic EG and stochastic OGDA under the standard unbiased bounded-variance stochastic oracle. The framework preserves the simplicity of the original algorithms, requires no additional oracle assumptions or variance-reduction techniques, and applies to both constrained and unconstrained convex--concave minimax problems.

\item We develop a unified analysis that first establishes last-iterate convergence in terms of the squared distance to the saddle point of the perturbed problem, and then converts this distance estimate into convergence guarantees for the restricted primal--dual gap and, in the unconstrained setting, the gradient norm.  

\item In the horizon-dependent setting, where the optimization horizon $T$ is known \emph{a priori}, we establish an $\mathcal{O}(T^{-1/4})$ last-iterate convergence rate for the restricted primal--dual gap and, in the unconstrained setting, the gradient norm. For compact feasible domains, the restricted gap coincides with the standard primal–dual gap, and in this case the rate improves upon the previous best $\widetilde{\mathcal{O}}(T^{-1/7})$ guarantee obtained under comparable assumptions, as detailed below in Table~\ref{tab:comparison}.

\item In the anytime setting, where the optimization horizon is unknown \emph{a priori}, we develop a diminishing-perturbation strategy under the standard bounded-variance stochastic oracle. For general closed convex feasible sets, we establish an $\mathcal{O}(T^{-1/5})$ last-iterate convergence rate for the restricted primal--dual gap for both PS-EG and PS-OGDA. Furthermore, in the unconstrained setting, we show that PS-EG admits a sharper $\mathcal{O}(T^{-1/4})$ last-iterate convergence rate in terms of the gradient norm.

\end{itemize}
 
\begin{table}[t]
\centering
\footnotesize
\renewcommand{\arraystretch}{0.92}

\begin{tabularx}{\linewidth}{
>{\raggedright\arraybackslash}p{0.14\linewidth}
>{\centering\arraybackslash}p{0.12\linewidth}
>{\centering\arraybackslash}p{0.13\linewidth}
>{\centering\arraybackslash}p{0.12\linewidth}
>{\centering\arraybackslash}p{0.11\linewidth}
X}
\toprule
Measure &
Method &
Rate &
Requires $T$? &
Domain &
Extra assumptions \\
\midrule

Gradient norm
&
\cite{cai2022stochastic}
&
$\mO(T^{-1/3})$
&
Yes
&
Unconstr.
&
Stronger oracle
\\
\midrule
Primal--dual gap
&
\cite{abe2025boosting}
&
$\widetilde \mO(T^{-1/7})$
&
Yes
&
Constr.
&
Bounded domain
\\
\midrule
Primal--dual gap
&
\cite{chen2024last}
&
$\mO(T^{-1/6})$
&
Yes
&
Constr.
&
Lipschitz Hessian \newline
+ bounded domain
 
 
\\
\midrule
 Gap /Gradient norm 
&
\textbf{PS-EG, PS-OGDA}
&
\textbf{$\mO(T^{-1/4})$}
\newline 
(Theorem~\ref{thm:rate_L2}, Corollary~\ref{coro:uncon})
& Yes 
&
\textbf{Both}
&
\textbf{---}
\\
\midrule
Gradient norm
&
\textbf{PS-EG}
&
\textbf{$\mO(T^{-1/4})$}
\newline
(Corollary~\ref{thm:rate_any_uncon_eg})
&
\multirow{2}{*}{\textbf{No}}
&
\textbf{Unconstr.}
&
\multirow{2}{*}{\textbf{---}}
\\
\cmidrule(lr){1-3}\cmidrule(lr){5-5}
Primal--dual gap
&
\textbf{PS-EG, PS-OGDA}
&
\textbf{$\mO(T^{-1/5})$}
\newline
(Theorem~\ref{thm:rate_any})
&
&
\textbf{Constr.}
&
\\

\bottomrule
\end{tabularx}

\caption{Comparison of last-iterate convergence guarantees for stochastic smooth convex–concave minimax optimization. For our methods, “Gap” denotes the restricted primal–dual gap $\mathcal{G}^R$ over compact comparison sets; it coincides with the standard primal–dual gap only in the compact-domain case with $\mathcal{B}_{\x}=\X$ and $\mathcal{B}_{\y}=\Y$. }
\label{tab:comparison}
\end{table}

\paragraph{Notation.} 
Throughout this paper, we use bold lowercase letters (e.g., $\x$, $\y$, $\z$) to denote vectors,
and calligraphic uppercase letters (e.g., $\X$, $\Y$) to denote sets.
For a set $\X$, we define its diameter as $
\diam(\X) \coloneqq  \sup_{\x,\y \in \X} \|\x - \y\|$. Given two sets $\X \subset \R^n$ and $\Y \subset \R^d$,
we denote their Cartesian product by $\X \times \Y \coloneqq  \{(\x,\y) : \x \in \X,\ \y \in \Y\}$. 
We use $\proj_{\X}(\z)$ to denote the orthogonal projection of a point $\z$ onto a closed convex set $\X$, defined as $\proj_{\X}(\z) \coloneqq \argmin_{\x \in \X} \|\x - \z\|$. 
Given a differentiable function $f:\R^n \times \R^d \to \R$,
we use $\nabla_{\x} f$ (resp., $\nabla_{\y} f$) to denote its partial gradient
with respect to the first (resp., second) argument. Finally, for two nonnegative quantities $a$ and $b$, the notation $a = \mO(b)$ implies that there exists a constant $C > 0$ such that $a \le Cb$, and $a = \Theta(b)$ implies that there exist constants $C_1, C_2 > 0$ such that $C_1 b \le a \le C_2 b$. We use $\tilde{\mO}(\cdot)$ to hide logarithmic factors. For a real number $a$, we use $\lceil a \rceil$ to denote the smallest integer greater than or equal to $a$. We let $\mathcal{N}_+ \coloneqq \{1,2,3,\ldots\}$ denote the set of positive integers.

\paragraph{Organization} The remainder of this paper is organized as follows. 
Section~\ref{sec:2} introduces the necessary preliminaries, including the assumptions and the notion of a stationary measure. 
In Section~\ref{sec:algorithms}, we present our proposed algorithms, PS-EG and PS-OGDA.  
Section~\ref{sec:main} presents the main last-iterate convergence guarantees for both the horizon-dependent and the anytime settings.
Section~\ref{sec:proof} develops the key ingredients of the analysis and proves the main results. Finally, Section~\ref{sec:6} concludes the paper.
 
\section{Preliminaries}\label{sec:2}
\subsection{Problem Setup}\label{sec:2_1}

In this section, we formalize the stochastic convex--concave minimax problem and delineate the structural assumptions required for our analysis. We consider the problem~\eqref{eq:prob} under the following standard conditions on the objective function $F$.

\begin{assumption}[Convex-concavity]\label{ass:cc}
The function $F(\x, \y)$ is convex in $\x$ for every fixed $\y \in \Y$, and concave in $\y$ for every fixed $\x \in \X$.
\end{assumption}

\begin{assumption}[Smoothness]\label{ass:l_smooth}
The function $F$ is continuously differentiable, and its partial gradients $\nabla_{\x} F$ and $\nabla_{\y} F$ are $L$-Lipschitz continuous. That is, for  all $(\x, \y), (\x', \y') \in \X \times \Y$, we have 
\begin{align*}
    \|\nabla_{\x} F(\x, \y) - \nabla_{\x} F(\x', \y')\| &\leq L (\|\x-\x'\| + \|\y-\y'\|), \\
    \|\nabla_{\y} F(\x, \y) - \nabla_{\y} F(\x', \y')\| &\leq L (\|\x-\x'\| + \|\y-\y'\|).
\end{align*}
\end{assumption}
For convex--concave minimax problems, it is mathematically advantageous to  leverage the variational inequality (VI) framework. Let $\u = [\x; \y] \in \R^{n+d}$ denote the concatenated variable and $\U = \X \times \Y$ the constraint set. We define the associated operator $V: \R^{n+d} \to \R^{n+d}$ as:
\begin{align*}
    V(\u) \coloneqq [ \nabla_{\x} F(\x, \y);\ -\nabla_{\y} F(\x, \y)].
\end{align*}
Under Assumptions~\ref{ass:cc} and~\ref{ass:l_smooth}, we know that $V$ is monotone and Lipschitz continuous~\citep{facchinei2003finite}. The problem of finding a saddle point of~\eqref{eq:prob} is then equivalent to solving the following VI:
\begin{equation}
    \text{Find } \u^\star \in \U \text{ such that } \langle V(\u^\star), \u - \u^\star \rangle \geq 0, \quad \forall \u \in \U. \label{eq:VI}
\end{equation}
In the unconstrained setting, where $\X = \R^n$ and $\Y = \R^d$, the VI problem reduces to solving  
\begin{equation*}
\text{Find } \u^\star \in \R^{n+d} \text{ such that }  V(\u^\star)=0.  
\end{equation*}
We refer the reader to~\citep{facchinei2003finite} for a comprehensive introduction to variational inequalities.    
To guarantee the tractability of the VI formulation and the well-posedness of subsequent analysis, we assume throughout that the solution set $\U^\star$ is non-empty.

\begin{assumption}[Existence of solution]\label{ass:exist}
The set of solutions to~\eqref{eq:VI},
\[
\U^\star = \{\u^\star\in\U:\ \langle V(\u^\star), \u-\u^\star\rangle \geq 0, \forall \u\in \U\},
\]
is nonempty. 
\end{assumption} 
 Throughout the analysis, we fix an arbitrary $\u^\star\in \U^\star$.
In many practical situations, the operator $V$ cannot be evaluated exactly, and we instead have access to unbiased stochastic estimates via a stochastic first-order oracle (SFO). Specifically, at each iteration $t$, we obtain a feedback of the form $\hat{V}_t = V(\u_t) + \bm{\zeta_t}$, where $\bm{\zeta_t} \in \R^{n+d}$ denotes an additive noise vector. We formalize this setup with the following assumption:

\begin{assumption}[Bounded variance]\label{ass:stoc}
The noise sequence $\{\bm{\zeta_t}\}_{t\ge 0}$ generated by the SFO satisfies
\begin{enumerate}[label=(\roman*)] 
    \item \textit{Zero mean:} $\E[\bm{\zeta_t} | \F_t] = 0$,
    \item \textit{Variance bound:} $\E[\|\bm{\zeta_t}\|^2 |\F_t] \leq \sigma^2$,
\end{enumerate}
for some $\sigma \geq 0$, where $\F_t$ denotes the natural filtration representing the history up to iteration $t$. The same conditions also hold for stochastic gradients evaluated at intermediate iterates, with the corresponding conditional expectations taken with respect to the appropriate filtration.
\end{assumption}
 
Assumption~\ref{ass:stoc} is standard in the stochastic minimax literature~\citep{gidelvariational, hsieh2019convergence}, where it is widely used to establish ergodic convergence guarantees~\citep{juditsky2011solving, mishchenko2020revisiting}. We note, however, that our analysis naturally extends to the generalized variance model~\citep{hsieh2020explore}, which allows the noise variance to grow quadratically with the distance to the solution set; i.e., $\E[\|\bm{\zeta_t}\|^2| \F_t] \leq \left(\sigma+\kappa \dist(\u_t,\U^\star)\right)^2$  and some $\kappa\geq 0$.  While this relaxation offers a more realistic alternative for settings with unbounded domains or heavy-tailed noise,  for clarity of exposition, we focus primarily on the simpler bounded variance case throughout this work. 
  
\subsection{Stationarity Measure}\label{sec:stat}
To evaluate the quality of a candidate solution, we next introduce the stationary measure. A standard notion for evaluating the optimality of a solution in minimax optimization is the primal-dual gap, defined as 
\[
 \mathcal{G}(\x,\y)=\max_{\y' \in \Y} F(\x,\y')-\min_{\x' \in \X} F(\x',\y).
\]
However, in the unconstrained setting where $\X=\R^n$ and $\Y=\R^d$, the primal-dual gap may be unbounded for any point that is not the exact optimal solution $\u^\star$. To circumvent this issue, we evaluate the gap over two suitably chosen compact subsets that contain at least one saddle point $\u^\star$, yielding the notion of \textit{Restricted primal-dual gap}.    
\begin{definition}[Restricted primal-dual gap]\label{def:standard_gap}
Let $\mathcal{B}_{\x} \subseteq \X$ and $\mathcal{B}_{\y} \subseteq \Y$ be compact sets such that they contain at least one saddle point $(\x^\star,\y^\star)$ of Problem~\eqref{eq:prob}. For any pair $(\x,\y) \in \X \times \Y$, the restricted primal-dual gap with respect to $\mathcal{B}_{\x} \times \mathcal{B}_{\y}$ is defined as:
\begin{align*}
     \mathcal{G}^R(\x,\y) \coloneqq \max_{\y' \in \mathcal{B}_{\y}} F(\x,\y') - \min_{\x' \in \mathcal{B}_{\x}} F(\x',\y).
\end{align*}
Furthermore, we say that $(\x,\y)$ is an  \emph{$\epsilon$-restricted saddle point} ($\epsilon$-RSP) with respect to
$\mathcal B_x\times\mathcal B_y$ if $\mathcal{G}^R(\x,\y)\le\epsilon$.
\end{definition}

For unconstrained problems, another widely used stationarity measure is the norm of the gradient operator $\|V(\u)\|$.
\begin{definition}\label{def:stat}
     For any pair $\u=(\x,\y) \in \X \times \Y$, we call $\u$ an \emph{$\epsilon$-stationary point} if
\[
\|V(\u)\|\le\epsilon.
\] 
\end{definition}  

\begin{remark}
The restricted primal--dual gap provides a unified stationarity measure for both constrained and unconstrained minimax problems. In particular, when the feasible sets $\X$ and $\Y$ are compact, choosing $\mathcal{B}_{\x}=\X$ and $\mathcal{B}_{\y}=\Y$ yields
\[
\mathcal{G}^R(\x,\y)=\mathcal{G}(\x,\y),
\]
so the restricted primal--dual gap coincides with the standard primal--dual gap.
\end{remark}
 
\section{Perturbation-Based Stochastic EG and OGDA}
\label{sec:algorithms}

Variational inequality (VI) formulations provide a natural framework for solving~\eqref{eq:prob}. Among first-order methods, the extragradient (EG) method~\citep{korpelevich1976extragradient} and optimistic gradient descent--ascent (OGDA)~\citep{mertikopoulos2019optimistic,mokhtari2020convergence} are particularly prominent. However, in stochastic settings, their last-iterate behavior is fundamentally problematic: even for simple bilinear problems, the iterates of vanilla stochastic EG (S-EG) and stochastic OGDA (S-OGDA) may diverge~\citep{chavdarova2019reducing,ryu2019ode,hsieh2020explore}.

To overcome these non-convergence issues, we introduce a simple perturbation scheme that induces strong monotonicity. Such perturbation or regularization techniques are also used in first-order methods to improve stability and enable last-iterate convergence for monotone problems (see, e.g.,~\citep{nesterov2013gradient,abe2024adaptively}). 
Specifically, for a parameter $\tau > 0$, we consider the quadratically perturbed objective
\begin{equation} \label{eq:prob_per}
  F_{\tau}(\x, \y) \coloneqq  F(\x,\y) + \frac{\tau}{2}\|\x\|^2 - \frac{\tau}{2}\|\y\|^2.
\end{equation}
By construction, $F_{\tau}$ is $\tau$-strongly convex in $\x$ and $\tau$-strongly concave in $\y$. When $\tau$ is chosen at an appropriate scale, the perturbed problem serves as a controlled approximation of the original objective, enabling convergence guarantees in terms of the primal–dual gap of the original problem.

Armed with the perturbed objective, we are ready to introduce our algorithm.
To this end, we consider the gradient mapping associated with the perturbed objective:
\[
W(\u) \coloneqq V(\u) + \tau \u,
\]
which satisfies $\tau$-strong monotonicity and is also Lipschitz continuous. 
For a time-varying perturbation sequence $\{\tau_t\}_{t\geq0}$, we associate the perturbation parameter with the iteration at which the stochastic oracle is queried. In particular, we use the convention
\[
\tau_{t+\frac{1}{2}} \coloneqq \tau_t, \qquad \forall t\geq 0.
\]
Applying stochastic extragradient (EG) and optimistic gradient descent--ascent (OGDA) to the perturbed operator yields the perturbed stochastic extragradient (PS-EG) and perturbed stochastic OGDA (PS-OGDA) methods, summarized in Algorithm~\ref{alg:perturbed_methods}.
In practice, the stochastic estimate $\hat{W}_t$ is constructed as $\hat{V}_t + \tau_t \u_t$, 
and similarly for the intermediate iterate. Hence no additional stochastic oracle calls are required beyond those used to return $\hat{V}_t$.

\begin{algorithm}[H] 
\makeatletter
\makeatother
 
\caption{Perturbed Stochastic Extragradient (PS-EG) and Optimistic GDA (PS-OGDA)}
\label{alg:perturbed_methods}

\KwIn{Initialization $\bu_0$ (and $\bu_{-\frac{1}{2}}$ for OGDA), learning rates $\eta_t$, perturbation parameters $\{\tau_t\}_{t=0}^T$}
\For{$t = 0, 1, \dots, T$}{
    \tcc{Access stochastic gradient estimates via SFO: $\hat{W}_t \coloneqq \hat{V}_t + \tau_t \bu_t$}
    
    \BlankLine
    \begingroup
    \everymath{\displaystyle}
    \rowcolors{1}{colorSEG}{colorSEG}
    \begin{tabular}{p{0.9\linewidth}}
    $\bu_{t+\frac{1}{2}} = \text{proj}_{\mathcal{U}}(\bu_t - \eta_t \hat{W}_t)$ \\
    $\bu_{t+1} = \text{proj}_{\mathcal{U}}(\bu_t - \eta_t \hat{W}_{t+\frac{1}{2}})$ \hfill \# PS-EG
    \end{tabular}
    \endgroup

    \BlankLine
    \begingroup
    \everymath{\displaystyle}
    \rowcolors{1}{colorSOGDA}{colorSOGDA}
    \begin{tabular}{p{0.9\linewidth}}
    $\bu_{t+\frac{1}{2}} = \text{proj}_{\mathcal{U}} (\bu_{t}-\eta_{t} \hat{W}_{t-\frac{1}{2}})$ \\
    $\bu_{t+1} = \text{proj}_{\mathcal{U}} (\bu_{t}-\eta_t \hat{W}_{t+\frac{1}{2}})$ \hfill \# PS-OGDA
    \end{tabular}
    \endgroup
}
\end{algorithm}

\begin{remark}[Connection to anchored methods]  
We briefly discuss the relation between the proposed perturbation approach and anchored (or regularized) methods~\citep{yoon2021accelerated, lee2021fast, cai2022accelerated, abe2025boosting, surina2026improved, cai2026last}. Anchored methods stabilize the dynamics by explicitly modifying the update rule, typically via an additional term that pulls the iterate toward a reference point (e.g., the initialization)~\citep{ryu2019ode, yoon2021accelerated}. A generic form of such updates is
\[ 
\u_{t+\frac{1}{2}} = \proj_{\mathcal{U}}\left( \u_t - \eta_t \hat{V}_t + \beta_t (\u_{\mathrm{ref}} - \u_t)\right), \quad 
\u_{t+1}  =  \proj_{\mathcal{U}}\left(\u_t - \eta_t \hat{V}_{t+\frac{1}{2}} + \beta_t (\u_{\mathrm{ref}} - \u_t)\right),
\]
where $\beta_t \ge 0$ is the anchoring coefficient, and $\beta_t = 0$ recovers S-EG.

At first sight, our PS-EG method may resemble an anchored update with $\u_{\rm ref}=0$. 
However, the two methods remain different. 
In anchored methods, the additional stabilization term typically acts on the current iterate $\u_t$, yielding a drift term of the form $-\beta_t \u_t$. 
By contrast, in the second step of PS-EG, the shrinkage term acts on the extrapolated point $\u_{t+\frac12}$ and arises implicitly from evaluating EG on the perturbed operator $W$.
Therefore, even when $\u_{\rm ref}=0$, PS-EG is not equivalent to an anchored method. 
The former preserves the exact EG structure applied to a modified problem, whereas the latter directly alters the update dynamics.

Despite recent progress, the theoretical understanding of anchored methods remains incomplete. In particular, their convergence behavior in constrained settings is largely unexplored. Existing variants that combine anchoring with variance reduction (e.g., STORM-type estimators) achieve ergodic convergence rates of $\mO(T^{-1/4})$~\citep{alacaoglu2025towards,alacaoglu2026solving}, while last-iterate guarantees for anchoring-based methods---especially under constraints---remains less understood. To our knowledge, the only work addressing this is~\citep{abe2025boosting}, which establishes a last-iterate rate of $\tilde{\mO}(T^{-1/7})$. Whether improved rates can be obtained for such methods remains an open question.
\end{remark}
 
\section{Main Results}
\label{sec:main}

We now present the main theoretical guarantees of the proposed perturbed stochastic methods. We consider two parameterization regimes. In Subsection~\ref{sec:T_rate}, we assume that the optimization horizon $T$ is known \emph{a priori}. This allows the perturbation parameter and the stepsizes to be selected according to the prescribed horizon $T$, yielding our best non-asymptotic last-iterate convergence rates. In Subsection~\ref{sec:any_rate}, we remove this assumption by employing diminishing perturbation and stepsize schedules, resulting in an anytime algorithm that does not require prior knowledge of the optimization horizon.
 
Before stating the main results, we introduce the compact sets used in the definition of the restricted primal--dual gap throughout the remainder of the paper. Their validity is justified later by establishing uniform mean-square boundedness of the iterates. Let
\[
C
\coloneqq
\sup_{t\ge0}
\mathbb E\!\left[\|\u_t-\u^\star\|^2\right],
\]
whose finiteness is established in Lemma~\ref{lem:bd_iter}. Accordingly, we define
\[
\mathcal B_{\x}
\coloneqq
\X\cap
\mathcal B(\x^\star,\sqrt C),
\qquad
\mathcal B_{\y}
\coloneqq
\Y\cap
\mathcal B(\y^\star,\sqrt C),
\]
and let $\mathcal G^R$ denote the restricted primal--dual gap over
$\mathcal B_{\x}\times\mathcal B_{\y}$.

In particular, if the feasible sets $\X$ and $\Y$ are compact, then
\[
\mathcal B_{\x}=\X,
\qquad
\mathcal B_{\y}=\Y,
\]
and consequently $\mathcal G^R$ coincides with the standard primal--dual gap.

\subsection{Horizon-Dependent Last-Iterate Convergence}
\label{sec:T_rate} 
Our first theorem establishes the last-iterate convergence rate when the optimization horizon is known in advance.
\begin{theorem}\label{thm:rate_L2}
Suppose Assumptions~\ref{ass:cc}, \ref{ass:l_smooth},
\ref{ass:exist}, and~\ref{ass:stoc} hold.
Let $T\ge1$. For either choice of the perturbation parameter and
stepsize sequence specified below, assume that $\tau_t\equiv \tau$ and $\sup_{t\ge0}\eta_t
\le
\min\left\{
\tfrac{1}{\sqrt{48}(L+\tau)},
\tfrac{1}{4\tau}
\right\}.$ 
Suppose further that $\sum_{t=0}^\infty \eta_t^2<\infty$, then the following last-iterate guarantees hold for both PS-EG and PS-OGDA.

\begin{enumerate}[label=(\roman*)] 

\item \textbf{Polynomial stepsizes.}
Suppose that $\tau=T^{-\delta}, 
\eta_t= (t+b)^{-\alpha},$ 
where $\delta>0$, $b\ge1$, and
$\alpha\in(\frac12,\,1-\delta)$.
Then, for all $T\ge b+1$,
\[
\mathbb{E} \left[\mathcal{G}^R(\mathbf{u}_T)\right]
\le
\mathcal{O} \left(
T^{-(\alpha-\delta)/2}
+
T^{-\delta}
\right).
\]
In particular, for any $0<\varepsilon<\frac{1}{12}$, choosing $\alpha=\frac34-3\varepsilon,
\delta=\frac14-\varepsilon$ 
yields
\[
\mathbb{E} \left[\mathcal{G}^R(\mathbf{u}_T)\right]
=
\mathcal{O} \left(T^{-1/4+\varepsilon}\right).
\]

\item \textbf{Regularization-aware harmonic stepsizes.}
Suppose that $\tau=T^{-1/4}, 
\eta_t=\tfrac{2}{\tau(t+b)},$ 
where $b=
\left\lceil
c_0
\max\left\{
\tfrac{2(L+\tau)}{\tau},
\tau^{-2}
\right\}
\right\rceil,
c_0>8$. 
Then
\[
\mathbb{E} \left[\mathcal{G}^R(\mathbf{u}_T)\right]
\le
\mathcal{O} \left(T^{-1/4}\right).
\]

\end{enumerate}
\end{theorem}

\begin{remark}\label{rmk:gap} (i)  For the polynomial stepsize, the condition $\delta < 1-\alpha$ ensures that the cumulative effect of 
the regularization dominates the stochastic fluctuations over the optimization 
horizon. Indeed, the leading term in the recursion scales as $\exp \left(-\tau \sum_{t=1}^T \eta_t\right)$, where $\sum_{t=1}^T \eta_t=\mathcal{O}(T^{1-\alpha}).$ 
For this factor to vanish and become negligible relative to the remaining 
error terms at horizon $T$, we require $\tau \sum_{t=1}^T \eta_t \to \infty$ 
as $T\to\infty$, which, under $\tau = \Theta(T^{-\delta})$, translates 
exactly to $\delta < 1-\alpha$.\\ (ii) The regularization-aware harmonic stepsize is   horizon-dependent. 
The choice $\eta_t=\tfrac{2}{\tau(t+b)}$ gives $\eta_t\tau=\frac{2}{t+b}$,
so the contraction term admits an exact telescoping bound. 
The offset $b=\Theta(\tau^{-2})$ ensures that $\sum_{t=0}^{\infty}\eta_t^2
=
\frac{4}{\tau^2}\sum_{t=0}^{\infty}\frac{1}{(t+b)^2}
\le
\frac{C}{\tau^2 b}
=
\Theta(1)$,
which is needed for the uniform mean-square boundedness of the iterates. 
This harmonic schedule avoids the strict boundary condition $\alpha+\delta<1$ that arises under the polynomial stepsize analysis.
\end{remark}
 
In the unconstrained setting, we also obtain a last-iterate guarantee in terms of the gradient norm.
\begin{corollary}[Unconstrained setting]\label{coro:uncon}
Under the assumptions of Theorem~\ref{thm:rate_L2}, in the unconstrained setting, the last iterate generated by PS-EG/PS-OGDA satisfies
	\begin{enumerate}[label=(\roman*)]
	\item \textbf{Polynomial stepsize.} With the polynomial stepsize stated in Theorem~\ref{thm:rate_L2}, we have
\[
\mathbb{E} [\|V(\mathbf{u}_T)\| ] \leq \mathcal{O}(T^{-1/4+\varepsilon})
\]
for any $\varepsilon>0$.
	\item \textbf{Regularization-aware harmonic stepsize.} 	With the regularization-aware harmonic stepsize stated in Theorem~\ref{thm:rate_L2}, we have
	\[
	\E[\|V(\u_T)\|]\leq\mO(T^{-1/4}).
	\]
	\end{enumerate}
\end{corollary}

These results are slightly slower than the best known last-iterate guarantee of $\mathcal{O}(T^{-1/3})$ in the unconstrained setting. However, achieving $\mathcal{O}(T^{-1/3})$ typically relies on stronger stochastic oracles (e.g., multiple samples per query) and variance reduction techniques~\cite{cai2022stochastic}.

\subsection{Anytime Last-Iterate Convergence}
\label{sec:any_rate}

We next consider the setting where the horizon $T$ is unavailable. Instead of selecting the perturbation parameter according to the horizon, we employ diminishing perturbation and stepsize schedules, resulting in an anytime algorithm.

\begin{theorem}\label{thm:rate_any}
Suppose Assumptions~\ref{ass:cc}, \ref{ass:l_smooth},
\ref{ass:exist}, and~\ref{ass:stoc} hold.
Let $\eta_t=\eta_0(t+b)^{-3/5}, 
\tau_t=\tau_0(t+b)^{-1/5},$ 
where $b\in \mathbb{N}_+$, and $\eta_0\tau_0b^{-4/5}\le1.$ 
Assume further that $\eta_{t}\leq \min\left\{\tfrac{\tau_t}{192\left(\tau_{t}-\tau_{t-1}\right)^2}, \tfrac{1}{96 \left(L+\tau_{t-1}\right) },\tfrac{1}{5\tau_t}\right\}$ holds for all $t\geq 1$.
Then, for both PS-EG and PS-OGDA, we have
\[
\mathbb E[\mathcal G^R(\u_t)]
=
\mathcal O(t^{-1/5})  \qquad \forall t\geq 1.
\]
In the unconstrained setting, we also have
\[
\mathbb{E} [\|V(\mathbf{u}_t)\| ] \leq \mathcal{O}(t^{-1/5}) \qquad \forall t\geq 1.
\]  
\end{theorem} 
For the unconstrained setting, the anytime convergence rate of PS-EG can be sharpened by a different reference-tracking argument.

\begin{theorem}[Sharper unconstrained anytime gradient-norm bound for PS-EG]
\label{thm:rate_any_uncon_eg}
Suppose that $\U=\mathbb R^{n+d}$, and Assumptions~\ref{ass:cc}, \ref{ass:l_smooth},
\ref{ass:exist}, and~\ref{ass:stoc} hold. Let the iterates be generated by PS-EG with
$\tau_t=\tau_0(t+b)^{-1/4}, 
\eta_t=\eta_0(t+b)^{-3/4},$
where $\eta_0,\tau_0>0$, $\eta_0\tau_0>1$, and $b\in\mathcal N_+$ is chosen sufficiently large so that $\eta_t\le \min\left\{\frac{\tau_t}{16(2L+\tau_t)^2}, 
 \frac{1}{2\tau_t}\right\}$ holds for all $t\geq 1$.  Then, for PS-EG, we have
\[
\E[\|V(\u_t)\| ]=\mO(t^{-1/4}), \qquad \forall t\geq 1.
\]
\end{theorem}
 
\begin{remark}
The proof of Theorem~\ref{thm:rate_any_uncon_eg} differs fundamentally from the analysis used throughout the rest of the paper. Rather than tracking the moving perturbed saddle point $\u_{\tau_t}^\star$, we introduce a deterministic reference trajectory generated by the same time-varying perturbed EG maps and establish a stochastic tracking bound between the two trajectories. Since, in the unconstrained setting, the gradient norm $\|V(\u)\|$ itself serves as the stationarity measure, this reference-tracking argument allows us to analyze the gradient norm directly, without first establishing convergence to the perturbed solution. Consequently, unlike our general proof framework, the analysis neither requires uniform mean-square boundedness of the stochastic iterates nor relies on converting a distance-to-solution estimate into a restricted primal--dual gap bound. Avoiding these ingredients bypasses one of the main technical bottlenecks in our analysis, allowing a more aggressive diminishing perturbation schedule and leading to the improved $\mathcal{O}(t^{-1/4})$ convergence rate.

This argument relies essentially on the fact that the gradient norm is an appropriate stationarity measure only in the unconstrained setting. Therefore, it does not extend directly to constrained problems, where convergence must instead be characterized through the restricted primal--dual gap. Moreover, the reference-tracking argument is specific to PS-EG and does not extend to PS-OGDA, whose recursive inequality contains additional terms arising from the variation of the perturbed operator across consecutive iterates.
\end{remark}
   
\section{Proof of Main Results}
\label{sec:proof}

In this section, we provide a high-level overview of the analysis underlying the main results.

Our analysis follows a unified two-stage framework. We first establish last-iterate convergence in terms of the squared distance to the saddle point of the perturbed problem, and then translate this distance estimate into convergence guarantees for the stationarity measures introduced in Section~\ref{sec:stat}, namely, the restricted primal--dual gap and, in the unconstrained setting, the norm of the gradient operator.

The first stage is developed in Section~\ref{sec:distance}. The horizon-dependent and the anytime analyses differ only in this stage: while the horizon-dependent analysis considers a fixed perturbed saddle point, the anytime analysis additionally requires tracking a time-varying perturbed saddle point induced by the diminishing perturbation parameter. The second stage is presented in Section~\ref{sec:dist_stat}, where the distance estimates are translated into convergence guarantees for the original problem. Combining these two ingredients yields the proofs of the main convergence theorems in Section~\ref{sec:proved_all}.

Finally, Section~\ref{sec:improved} develops a separate analysis for the unconstrained PS-EG method. Rather than following the above two-stage framework, it introduces a reference-tracking argument that yields a sharper anytime convergence rate in terms of the gradient norm.
\subsection{Distance Analysis}
\label{sec:distance}

In this subsection, we establish last-iterate convergence in terms of the squared distance to the saddle point of the perturbed problem. This constitutes the main technical component of the paper and serves as the foundation for all subsequent convergence guarantees for the original problem.

Our analysis begins by deriving recursive inequalities that characterize the one-step progress of the stochastic iterates. The quadratic perturbation renders the gradient operator $\tau$-strongly monotone in the horizon-dependent setting and $\tau_t$-strongly monotone in the anytime setting, thereby introducing a contraction into the recursion. By carefully balancing this contraction against the stochastic error, we obtain non-asymptotic convergence bounds for the distance to the perturbed saddle point.

Before proceeding, we summarize several basic properties of the perturbed saddle point that will be used repeatedly throughout the analysis. In particular, the second property quantifies the variation of the perturbed solution as the perturbation parameter changes, which is essential for the anytime analysis.

\begin{lemma}[Properties of perturbed saddle points]\label{lem:per_sol} Under Assumptions~\ref{ass:cc}, \ref{ass:l_smooth}, and~\ref{ass:exist}, for any $\tau > 0$,  if $\u_{\tau}^\star$ solves the problem~\eqref{eq:prob_per}, then we have
	\begin{enumerate}
		\item For any $\tau > 0$, the perturbed solution $\u_\tau^\star$ satisfies
		\[
		\|\u_\tau^\star\|\le \|\u^\star\|.
		\] 
		\item For any $0<\tau'<\tau$, we have
		\[
		\|\u_\tau^\star-\u_{\tau'}^\star\|
		\le \frac{|\tau-\tau'|}{\tau}\|\u_{\tau'}^\star\|\leq \frac{|\tau-\tau'|}{\tau}\|\u^\star\|.
		\]  
	\end{enumerate}
\end{lemma}

\subsubsection{Recursive inequalities} 
In this section, we derive recursive inequalities for the EG and OGDA methods that characterize the one-step progress of the stochastic updates. Such recursions are standard in the analysis of stochastic extragradient methods~\cite{hsieh2019convergence,hsieh2020explore}. However, existing analyses do not account for the perturbation terms required in our framework and typically discard negative quadratic terms that play a crucial role in establishing contraction. We therefore derive refined recursive inequalities that explicitly capture these effects.  
 
We first derive the recursive inequality for PS-EG.
\begin{lemma}
	\label{lem:key_eg}
	Let $\{\u_t\}_{t=1}^T,\{\u_{t+\frac{1}{2}}\}_{t=1}^T$ be the iterates generated by the PS-EG. For all $t=1,2,\cdots$ and $\u \in \U$ that are $\F_{t+\frac{1}{2}}$-measurable, we have 
	\begin{align*}
		   \E[ \|\u_{t+1}-\u \|^2] 
        \leq\ & \E[ \|\u_{t }-\u \|^2]-2\eta_t \E [\langle W(\u_{t+\frac{1}{2}}), \u_{t+\frac{1}{2}}-\u   \rangle]\\
		&- \left(1-6\eta_t^2 \left(L+\tau_t\right)^2\right)\E[ \|\u_{t+\frac{1}{2}}-\u_t\|^2] + 6\eta_t^2\sigma^2.
	\end{align*}  
\end{lemma}

\begin{proof} 
	Applying Lemma~\ref{lem:proj_key} with 
	\[
	(\x,\y_1,\y_2,\x_1^+,\x_2^+) 
	\leftarrow 
	(\u_t, \eta_t \hat{W}_t, \eta_t \hat{W}_{t+\frac{1}{2}}, \u_{t+\frac{1}{2}}, \u_{t+1})
	\]
	and setting $\p=\u$, we obtain
	\begin{align*}
		\|\u_{t+1}-\u\|^2
		\leq \|\u_t-\u\|^2
		-2\eta_t\langle \hat{W}_{t+\frac{1}{2}}, \u_{t+\frac{1}{2}}-\u \rangle
		+\eta_t^2\|\hat{W}_t-\hat{W}_{t+\frac{1}{2}}\|^2
		-\|\u_{t+\frac{1}{2}}-\u_t\|^2.
	\end{align*}
	Taking expectation on both sides yields
	\begin{align}\label{eq:main_constrained_eg}
		\E[\|\u_{t+1}-\u\|^2]
		\leq\ & \E[\|\u_t-\u\|^2]
		-2\eta_t \E[\langle \hat{W}_{t+\frac{1}{2}}, \u_{t+\frac{1}{2}}-\u \rangle]   + \eta_t^2 \E[\|\hat{W}_t-\hat{W}_{t+\frac{1}{2}}\|^2]\notag\\
		&- \E[\|\u_{t+\frac{1}{2}}-\u_t\|^2].
	\end{align}
	We next bound each term. Let $\mathcal{F}_t$ denote the filtration up to time $t$, and define $\E_t[\cdot] = \E[\cdot \mid \mathcal{F}_t]$.
	For the inner product term, by Assumption~\ref{ass:stoc}, we have
	\begin{align*}
		\E[\langle \hat{W}_{t+\frac{1}{2}}, \u_{t+\frac{1}{2}}-\u \rangle]
		= \E[\E_{t+\frac{1}{2}}[\langle \hat{W}_{t+\frac{1}{2}}, \u_{t+\frac{1}{2}}-\u \rangle]]  = \E[\langle W(\u_{t+\frac{1}{2}}), \u_{t+\frac{1}{2}}-\u \rangle].
	\end{align*}
For the quadratic term, we have
\begin{align*}
	\E[\|\hat{W}_t-\hat{W}_{t+\frac{1}{2}}\|^2]
	=\ & \E[\|W(\u_t)-W(\u_{t+\frac{1}{2}}) + \bm{\zeta_t}  - \bm{\zeta}_{t+\frac{1}{2}}\|^2] \\
	\leq\ & 3\E[\|W(\u_t)-W(\u_{t+\frac{1}{2}})\|^2]
	+  3\E[\E_t[\|\bm{\zeta_t} \|^2]]+3\E[\E_{t+\frac{1}{2}}[\|\bm{\zeta_}{t+\frac{1}{2}}\|^2]] \\
	=\ & 3\E[\|V(\u_t)+\tau_t \u_t-V(\u_{t+\frac{1}{2}}) -\tau_{t+\frac{1}{2}}\u_{t+\frac{1}{2}}\|^2]+ 6\sigma^2\\
	\leq\ & 6\left(L+\tau_t\right)^2 \E[\|\u_t-\u_{t+\frac{1}{2}}\|^2] + 6 \left(\tau_t-\tau_{t+\frac{1}{2}} \right)^2 \E[\|\u_{t}\|^2]+6\sigma^2,
\end{align*}
	where the first inequality follows from $\|a+b+c\|^2 \leq 3(\|a\|^2+\|b\|^2+\|c\|^2)$, 
	and the second follows from Assumption~\ref{ass:stoc} and $\tau_t\geq \tau_{t+\frac{1}{2}}$.
	
	Substituting the above bounds into~\eqref{eq:main_constrained_eg} yields the desired recursive inequality.
\end{proof}

A similar recursion can be established for the PS-OGDA update. Compared with PS-EG, the bound additionally involves the variation of the gradient operator $W$ across consecutive iterates, which is essential for controlling the optimistic correction term.

\begin{lemma}
	\label{lem:key_ogda}
	Let $\{\u_t\}_{t=1}^T,\{\u_{t-\frac{1}{2}}\}_{t=1}^T$ be the iterates generated by the PS-OGDA. For all $t=1,2,\cdots$ and $\u \in \U $ that are $\F_{t+\frac{1}{2}}$-measurable, we have 
	\begin{align*} 
	&\E[\|\u_{t+1}-\u\|^2]+3\eta_t^2 \E[\|W(\u_{t+\frac{1}{2}})-W(\u_{t-\frac{1}{2}})\|^2] \\
    \leq\ & \E[ \|\u_{t}-\u\|^2]+72\eta_{t-1}^2\eta_t^2\left(L+\tau_{t-1}\right)^2\E [\| W(\u_{t-\frac{1}{2}})-W(\u_{t-\frac{3}{2}})  \|^2]\\
    &
	-2\eta_t \E[\langle W(\u_{t+\frac{1}{2}}), \u_{t+\frac{1}{2}}-\u \rangle] -\left(1-24 \eta_t^2\left(L+\tau_{t-1}\right)^2\right) \E[ \| \u_t -\u_{t+\frac{1}{2}}\|^2]  \\
	&
	+12\left(\tau_{t}-\tau_{t-1}\right)^2\eta_t^2 \E[\|\u_{t-\frac{1}{2}}\|^2] + \left(144\eta_{t-1}^2\left(L+\tau_{t-1}\right)^2+6\right)\eta_t^2\sigma^2.
\end{align*}
\end{lemma}

\begin{proof} 
	Applying Lemma~\ref{lem:proj_key} with
	\[
	(\x,\y_1,\y_2,\x_1^+,\x_2^+) 
	\leftarrow 
	(\u_{t}, \eta_t \hat{W}_{t-\frac{1}{2}}, \eta_t \hat{W}_{t+\frac{1}{2}}, \u_{t+\frac{1}{2}}, \u_{t+1})
	\]
	and setting $\p=\u$, we obtain
	\begin{align*}
		\|\u_{t+1}-\u\|^2
		\leq \|\u_{t}-\u\|^2
		-2\eta_t \langle \hat{W}_{t+\frac{1}{2}}, \u_{t+\frac{1}{2}}-\u \rangle
		+\eta_t^2 \|\hat{W}_{t+\frac{1}{2}}-\hat{W}_{t-\frac{1}{2}}\|^2
		-\|\u_t-\u_{t+\frac{1}{2}}\|^2.
	\end{align*}
	Taking expectation on both sides yields
	\begin{align}\label{eq:main_constrained_ogda}
		\E[\|\u_{t+1}-\u\|^2]
		\leq\ & \E[ \|\u_{t}-\u\|^2]
		-2\eta_t \E[\langle \hat{W}_{t+\frac{1}{2}}, \u_{t+\frac{1}{2}}-\u \rangle]  + \eta_t^2 \E[\|\hat{W}_{t+\frac{1}{2}}-\hat{W}_{t-\frac{1}{2}}\|^2]\notag\\
		&- \E[\|\u_t-\u_{t+\frac{1}{2}}\|^2].
	\end{align}
	We next bound each term. Let $\mathcal{F}_t$ denote the filtration up to time $t$, and define $\E_t[\cdot] = \E[\cdot \mid \mathcal{F}_t]$.
	
	For the inner product term, by Assumption~\ref{ass:stoc}, we have
	\begin{align}\label{eq:inn_bd}
		\E[\langle \hat{W}_{t+\frac{1}{2}}, \u_{t+\frac{1}{2}}-\u \rangle]=\E[\E_{t+\frac{1}{2}}[\langle \hat{W}_{t+\frac{1}{2}}, \u_{t+\frac{1}{2}}-\u \rangle]]
		= \E[\langle W(\u_{t+\frac{1}{2}}), \u_{t+\frac{1}{2}}-\u \rangle].
	\end{align}
	
	For the quadratic term, we have
	\begin{align}\label{eq:diff_grad}
		\E[\|\hat{W}_{t+\frac{1}{2}}-\hat{W}_{t-\frac{1}{2}}\|^2]
		=\ & \E  [\|W(\u_{t+\frac{1}{2}})-W(\u_{t-\frac{1}{2}})+  \bm{\zeta}_{t+\frac{1}{2}}-\bm{\zeta}_{t-\frac{1}{2}}\|^2] \notag\\
		\leq\ & 3\E  [\|W(\u_{t+\frac{1}{2}})-W(\u_{t-\frac{1}{2}})\|^2 +  3\E[\E_{t+\frac{1}{2}}[\|\bm{\zeta}_{t+\frac{1}{2}}\|^2]]+3\E[\E_{t-\frac{1}{2}}[\|\bm{\zeta}_{t-\frac{1}{2}}\|^2]] \notag\\
		\leq\ &  3\E [\|W(\u_{t+\frac{1}{2}})-W(\u_{t-\frac{1}{2}})\|^2] +  6\sigma^2.
	\end{align}
	where we used $\|a+b+c\|^2 \le 3(\|a\|^2+\|b\|^2+\|c\|^2)$ and Assumption~\ref{ass:stoc}.
	It remains to control $\E[\|W(\u_{t+\frac{1}{2}})-W(\u_{t-\frac{1}{2}})\|^2]$. By Lipschitz continuity of $V$ and the fact that $\tau_{t-\frac{1}{2}}\geq \tau_{t+\frac{1}{2}}$, we have
	\begin{align*}
		&\E[\|W(\u_{t+\frac{1}{2}})-W(\u_{t-\frac{1}{2}})\|^2]\\
		=\ & 	\E[\|V(\u_{t+\frac{1}{2}})+\tau_{t+\frac{1}{2}}\u_{t+\frac{1}{2}}-V(\u_{t-\frac{1}{2}})-\tau_{t-\frac{1}{2}}\u_{t-\frac{1}{2}}\|^2]\\
		\leq\ & 2\left(L+\tau_{t-\frac{1}{2}}\right)^2 \E[\|\u_{t+\frac{1}{2}}-\u_{t-\frac{1}{2}}\|^2]+2\left(\tau_{t+\frac{1}{2}}-\tau_{t-\frac{1}{2}}\right)^2 \E[\|\u_{t-\frac{1}{2}}\|^2]\\
		\leq\ & 4\left(L+\tau_{t-\frac{1}{2}}\right)^2 \E[\|\u_{t+\frac{1}{2}}-\u_{t}\|^2]
		+  4\left(L+\tau_{t-\frac{1}{2}}\right)^2 \E[\|\u_{t}-\u_{t-\frac{1}{2}}\|^2]\\
        &+2\left(\tau_{t+\frac{1}{2}}-\tau_{t-\frac{1}{2}}\right)^2 \E[\|\u_{t-\frac{1}{2}}\|^2].
	\end{align*}
	 
	Using the update rule and non-expansiveness of projection, we further bound
	\[
	\E[\|\u_{t-\frac{1}{2}}-\u_{t}\|^2]
	\leq \eta_{t-1}^2 \E[\|\hat{W}_{t-\frac{1}{2}}-\hat{W}_{t-\frac{3}{2}}\|^2]\leq 3\eta_{t-1}^2\E [\| W(\u_{t-\frac{1}{2}})-W(\u_{t-\frac{3}{2}})  \|^2]+6\eta_{t-1}^2 \sigma^2.
	\]
	Therefore, we get 
	\begin{align*}
		&\E[\|W(\u_{t+\frac{1}{2}})-W(\u_{t-\frac{1}{2}})\|^2] \\
		\leq\ & 8\left(L+\tau_{t-\frac{1}{2}}\right)^2 \E[\|\u_{t+\frac{1}{2}}-\u_{t}\|^2]+24 \eta_{t-1}^2\left(L+\tau_{t-\frac{1}{2}}\right)^2 \E [\| W(\u_{t-\frac{1}{2}})-W(\u_{t-\frac{3}{2}})  \|^2]\\
		&+4\left(\tau_{t+\frac{1}{2}}-\tau_{t-\frac{1}{2}}\right)^2 \E[\|\u_{t-\frac{1}{2}}\|^2] 
	 - \E[\|W(\u_{t+\frac{1}{2}})-W(\u_{t-\frac{1}{2}})\|^2] +  48\eta_{t-1}^2\left(L+\tau_{t-\frac{1}{2}}\right)^2\sigma^2.
	\end{align*}
	This, together with the bounds~\eqref{eq:main_constrained_ogda},~\eqref{eq:inn_bd}, and~\eqref{eq:diff_grad} we obtain that 
	\begin{align*} 
		&\E[\|\u_{t+1}-\u\|^2]\\
		\leq\ & \E[ \|\u_{t}-\u\|^2]
		-2\eta_t \E[\langle W(\u_{t+\frac{1}{2}}), \u_{t+\frac{1}{2}}-\u \rangle] -\left(1-24 \eta_t^2\left(L+\tau_{t-\frac{1}{2}}\right)^2\right) \E[ \| \u_t -\u_{t+\frac{1}{2}}\|^2] \notag\\
		&+72\eta_{t-1}^2\eta_t^2\left(L+\tau_{t-\frac{1}{2}}\right)^2\E [\| W(\u_{t-\frac{1}{2}})-W(\u_{t-\frac{3}{2}})  \|^2]- 3\eta_t^2 \E[\|W(\u_{t+\frac{1}{2}})-W(\u_{t-\frac{1}{2}})\|^2]  \notag\\
		&+12\left(\tau_{t+\frac{1}{2}}-\tau_{t-\frac{1}{2}}\right)^2\eta_t^2 \E[\|\u_{t-\frac{1}{2}}\|^2] + \left(144\eta_{t-1}^2\left(L+\tau_{t-\frac{1}{2}}\right)^2+6\right)\eta_t^2\sigma^2.
	\end{align*}
	Rearranging it and using the fact that $\tau_t=\tau_{t+\frac{1}{2}}, \forall t\geq 0$ finishes the proof.  
\end{proof}

\subsubsection{Horizon-dependent Distance Convergence}
We begin with the horizon-dependent setting, in which the perturbation parameter remains fixed during the entire optimization process. Based on the recursive inequalities derived above, we establish a non-asymptotic last-iterate convergence guarantee for the squared distance to the perturbed saddle point.
  \begin{theorem}\label{thm:pert_rate}
 	Under the setting of Theorem~\ref{thm:rate_L2}, the last iterate generated by PS-EG/PS-OGDA satisfies
 	\begin{enumerate}[label=(\roman*)]
 		\item \textbf{Polynomial stepsize.} With the polynomial stepsize stated in Theorem~\ref{thm:rate_L2},  
 	 there exist constants $c,C>0$, independent of $T$, such that for all $T\ge 1+b$,
 	\[ \E[\|\u_T-\u^\star_\tau\|^2] 
 	\leq 
 	\E[\|\u_0-\u^\star_\tau\|^2]  
 	\exp \left(-cT^{1-\alpha-\delta}\right)
 	+
 	C T^{-\alpha+\delta}, 
 	\quad \text{for } \tfrac{1}{2}<\alpha<1-\delta,
 	\]
 	which implies that $\E[\|\u_T-\u_{\tau}^\star\|^2] \leq \mathcal{O}(T^{-\alpha+\delta}), 
 	\quad \text{for } \tfrac{1}{2}<\alpha<1-\delta$.
 
 	\item \textbf{Regularization-aware harmonic stepsize.} 	With the regularization-aware harmonic stepsize stated in Theorem~\ref{thm:rate_L2}, we have that there exist constants $C>0$, independent of $T$, such that for all $T\ge 1$,
 	\[
 	\E[\|\u_T-\u_\tau^\star\|^2]
 	\le
 	 \E[\|\u_0-\u_\tau^\star\|^2]\frac{Cb^2}{(T+b)^2}
 	+
\frac{9C\sigma^2T}{\tau^2(T+b)^2},
 	\]
 		which implies that $\E[\|\u_T-\u_{\tau}^\star\|^2] \leq \mathcal{O}(T^{-1/2})$.
 	\end{enumerate}
 \end{theorem}

\begin{proof} 
We establish the convergence rates of PS-EG and PS-OGDA based on the recursive inequalities in Lemmas~\ref{lem:key_eg} and~\ref{lem:key_ogda}.  

\begin{enumerate}[label=(\roman*)]
\item \textbf{PS-EG:}
By the $\tau$-strong monotonicity of $W$ and Young’s inequality, we have 
\begin{align*}
\langle W(\u_{t+\frac{1}{2}}), \u_{t+\frac{1}{2}}-\u_{\tau}^\star\rangle
\geq \tau\|\u_{t+\frac{1}{2}}-\u_{\tau}^\star\|^2
\geq \frac{\tau}{2}\|\u_t-\u_{\tau}^\star\|^2
- \tau\|\u_{t+\frac{1}{2}}-\u_t\|^2.
\end{align*}

Combining this with Lemma~\ref{lem:key_eg} (with $\u=\u_{\tau}^\star$), we obtain
\begin{align}\label{eq:eg_1}
&\E[\|\u_{t+1}-\u_{\tau}^\star\|^2] \notag\\
\leq\ & (1-\eta_t\tau)\E[\|\u_t-\u_{\tau}^\star\|^2]  
- (1-6\eta_t^2 (L+\tau)^2-2\eta_t\tau)\E[\|\u_{t+\frac{1}{2}}-\u_t\|^2]
+ 6\eta_t^2\sigma^2 \notag\\
\leq\ & (1-\eta_t\tau)\E[\|\u_t-\u_{\tau}^\star\|^2] + 6\eta_t^2\sigma^2,
\end{align}
where the last inequality follows from the stepsize condition 
$6\eta_t^2(L+\tau)^2\leq \tfrac{1}{2}$ and $2\eta_t\tau \leq \tfrac{1}{2}$.
\begin{enumerate}[label=(\roman*)] 
	\item \textbf{Polynomial stepsize:} If the stepsize satisfies $\eta_t = (t+b)^{-\alpha}$ with $\frac{1}{2} < \alpha < 1$ and $b\geq 1$, 
	and the perturbation parameter scales as $\tau = T^{-\delta}$ for some $\delta > 0$.
The resulting recursion is in the standard form of Lemma~\ref{lem:chung_mod}. By applying this lemma, we obtain that, there exist constants $c,C>0$, independent of $T$, such that for all $T\ge 1+b$,
\[
\E[\|\u_T-\u_{\tau}^\star\|^2] 
\leq 
\E[\|\u_0-\u_{\tau}^\star\|^2] 
\exp \left(-cT^{1-\alpha-\delta}\right)
+
C T^{-\alpha+\delta}, 
\quad \text{for } \tfrac{1}{2}<\alpha<1-\delta,
\]
which implies that $\E[\|\u_T-\u_{\tau}^\star\|^2] \leq \mathcal{O}(T^{-\alpha+\delta}), 
\quad \text{for } \tfrac{1}{2}<\alpha<1-\delta$.
\item  \textbf{Regularization-aware harmonic stepsize.}  If the stepsize satisfies $\eta_t=2/(\tau(t+b))$ with $b=\left\lceil c_0\max\left\{  2( L+\tau)/\tau,\tau^{-2}\right\}\right\rceil$ and $c_0>8$, and the perturbation parameter scales as $\tau = T^{-1/4}$.  With $\eta_t=2/(\tau(t+b))$, the resulting recursion is in the standard form of Lemma~\ref{lem:harmonic_chung}.  
By applying this lemma, we obtain that, there exist constants $C>0$, independent of $T$, such that for all $T\ge 1$,
\[
\E[\|\u_T-\u_\tau^\star\|^2]
\le
 \E[\|\u_0-\u_\tau^\star\|^2]\frac{Cb^2}{(T+b)^2}
+
\frac{6C\sigma^2T}{\tau^2(T+b)^2},
\]
	which implies that $\E[\|\u_T-\u_{\tau}^\star\|^2] \leq\mathcal{O}(T^{-1/2})$.
\end{enumerate}
 
\item \textbf{PS-OGDA:} 
By $\tau$-strong monotonicity and Young’s inequality, we have  
\begin{align*}
\langle W(\u_{t+\frac{1}{2}}), \u_{t+\frac{1}{2}}-\u_{\tau}^\star \rangle
\geq \tau\|\u_{t+\frac{1}{2}}-\u_{\tau}^\star\|^2
\geq \frac{\tau}{2}\|\u_{t}-\u_{\tau}^\star\|^2
- \tau\|\u_t-\u_{t+\frac{1}{2}}\|^2.
\end{align*}

Combining this with Lemma~\ref{lem:key_ogda}, we obtain
\begin{align} \label{eq:og_1}
      &\E[\|\u_{t+1}-\u_{\tau}^\star\|^2  +3\eta_t^2  \|W(\u_{t+\frac{1}{2}})-W(\u_{t-\frac{1}{2}})\|^2]  \notag\\
      \leq\ & \left(1-\eta_t \tau\right)\E[\|\u_{t}-\u_{\tau}^\star\|^2] +72\eta_{t-1}^2\eta_t^2\left( L+\tau\right)^2 \E [\| W(\u_{t-\frac{1}{2}})-W(\u_{t-\frac{3}{2}})  \|^2] \notag\\
      &-\left(1-24 \eta_t^2\left( L+\tau\right)^2-2\eta_t \tau\right) \E[ \| \u_t -\u_{t+\frac{1}{2}}\|^2]  + \left(144\eta_{t-1}^2\left( L+\tau\right)^2+6\right)\eta_t^2\sigma^2 \notag\\
      =\ &\left(1-\eta_t \tau\right)\left(\E[\|\u_{t}-\u_{\tau}^\star\|^2+3\eta_{t-1}^2  \| W(\u_{t-\frac{1}{2}})-W(\u_{t-\frac{3}{2}})  \|^2]\right) \notag\\
      &-3\eta_{t-1}^2\left(1-\eta_t\tau-24\eta_t^2\left(L+\tau\right)^2\right)\E[\| W(\u_{t-\frac{1}{2}})-W(\u_{t-\frac{3}{2}})  \|^2] \notag\\
      &-\left(1-24\eta_t^2\left( L+\tau\right)^2-2\eta_t \tau\right) \E[ \| \u_t -\u_{t+\frac{1}{2}}\|^2]+ \left(144\eta_{t-1}^2\left( L+\tau\right)^2+6\right)\eta_t^2\sigma^2 \notag\\
      \leq\ & \left(1-\eta_t \tau\right)\left(\E[\|\u_{t}-\u_{\tau}^\star\|^2+3\eta_{t-1}^2  \| W(\u_{t-\frac{1}{2}})-W(\u_{t-\frac{3}{2}})  \|^2]\right) + 9\eta_t^2\sigma^2,
\end{align}
where the last inequality uses the stepsize condition  
$24\eta_t^2 ( L+\tau)^2 \leq \tfrac{1}{2} \leq 1-2\eta_t \tau$.

To capture both the distance and the gradient variation, we define the Lyapunov function
\[
\Phi_t \coloneqq \E[\|\u_{t}-\u_{\tau}^\star\|^2  +3\eta_{t-1}^2  \|W(\u_{t-\frac{1}{2}})-W(\u_{t-\frac{3}{2}})\|^2].
\]
with $\u_{-\frac{3}{2}}=\u_{-\frac{1}{2}}$.
\begin{enumerate}
	\item \textbf{Polynomial stepsize:} If the stepsize satisfies $\eta_t = (t+b)^{-\alpha}$ with $\frac{1}{2} < \alpha < 1$ and $b\geq 1$, 
and the perturbation parameter scales as $\tau = T^{-\delta}$ for some $\delta > 0$.
The resulting recursion is in the standard form of Lemma~\ref{lem:chung_mod}. By applying this lemma, we obtain that, there exist constants $c,C>0$, independent of $T$, such that for all $T\ge 1+b$,
\[ \E[\Phi_T] 
\leq 
\E[\Phi_0]  
\exp \left(-cT^{1-\alpha-\delta}\right)
+
C T^{-\alpha+\delta}, 
\quad \text{for } \tfrac{1}{2}<\alpha<1-\delta,
\]
which implies that $\E[\|\u_T-\u_{\tau}^\star\|^2] \leq\mathcal{O}(T^{-\alpha+\delta}), 
\quad \text{for } \tfrac{1}{2}<\alpha<1-\delta$.
 
\item  \textbf{Regularization-aware harmonic stepsize.}  If the stepsize satisfies $\eta_t=2/(\tau(t+b))$ with $b=\left\lceil c_0\max\left\{ 2( L+\tau)/\tau,\tau^{-2}\right\}\right\rceil$ and $c_0>8$, and the perturbation parameter scales as $\tau = T^{-1/4}$.  With $\eta_t=2/(\tau(t+b))$, the resulting recursion is in the standard form of Lemma~\ref{lem:harmonic_chung}.  
By applying this lemma, we obtain that, there exist constants $C>0$, independent of $T$, such that for all $T\ge 1$,
	\[
\E[\Phi_T]
\le
 \E[\Phi_0]\frac{Cb^2}{(T+b)^2}
+
\frac{9C\sigma^2T}{\tau^2(T+b)^2},
\]
	which implies that $\E[\|\u_T-\u_{\tau}^\star\|^2] \leq \mathcal{O}(T^{-1/2})$.
\end{enumerate}
This finishes the proof.
\end{enumerate}
\end{proof}

\subsubsection{Anytime Distance Convergence}
We next consider the anytime setting, where the perturbation parameter decreases over time. In contrast to the horizon-dependent analysis, the perturbed saddle point is no longer fixed, and the proof must additionally control the drift of the moving center. Nevertheless, by combining the recursive inequalities with Lemma~\ref{lem:per_sol}, we obtain the following distance convergence guarantee.
 \begin{theorem}\label{thm:pert_any}
 	Under the setting of Theorem~\ref{thm:rate_any}, we have
 	\[ \E[\|\u_t-\u^\star_{\tau_t}\|^2] 
 	\leq  \mathcal{O}(t^{-2/5}), 
 	\quad \forall t\geq 1.
    \]
 \end{theorem}
\begin{proof}
    We establish the convergence rate for PS-EG and PS-OGDA based on the recursive inequalities in Lemmas~\ref{lem:key_eg} and~\ref{lem:key_ogda}.

\begin{enumerate}[label=(\roman*)]
\item \textbf{PS-EG:}  
By the $\tau_{t+\frac{1}{2}}$-strong monotonicity of $W$ and Young’s inequality, we have 
\begin{align*}
	 \langle W(\u_{t+\frac{1}{2}}),
	\u_{t+\frac{1}{2}}-\u_{\tau_{t}}^\star\rangle 
	=\ & 	\langle W(\u_{t+\frac{1}{2}}),
	\u_{t+\frac{1}{2}}-\u_{\tau_{t+\frac{1}{2}}}^\star\rangle \geq
	\tau_{t}
	\|\u_{t+\frac{1}{2}}-\u_{\tau_{t+\frac{1}{2}}}^\star\|^2 \\
	\geq\ &
	\frac{\tau_{t}}{2}
	\|\u_t-\u_{\tau_{t}}^\star\|^2
	-\tau_{t }
	\|\u_{t+\frac{1}{2}}-\u_t\|^2  .
\end{align*}
Combining this bound with Lemma~\ref{lem:key_eg} (with $\u=\u_{\tau_t}^\star$), we obtain
\begin{align}\label{eq:mid_1}
&\E[\|\u_{t+1}-\u_{\tau_t}^\star\|^2] \notag\\
\leq\ & \left(1- \eta_t\tau_{t} \right)\E[\|\u_t-\u_{\tau_t}^\star\|^2]  - (1-6\eta_t^2 \left(L+\tau_t\right)^2-2\eta_t\tau_{t })\E[\|\u_{t+\frac{1}{2}}-\u_t\|^2] +6\eta_t^2\sigma^2 \notag \\
\leq\ & \left(1- \eta_t\tau_{t} \right)\E[\|\u_t-\u_{\tau_t}^\star\|^2] +6\eta_t^2\sigma^2 ,
\end{align}
where the last inequality follows from the stepsize condition 
$6\eta_t^2\left(L+\tau_t\right)^2\leq \tfrac{1}{2} \leq 1-2\eta_t\tau_t$. 

On the other hand, by Jensen's inequality, for any $\gamma_t\geq 0$, we have 
\begin{align}\label{eq:mid_2}
	\E[\|\u_{t+1}-\u_{\tau_{t+1}}^\star\|^2]\leq\ &\left(1+\gamma_t\right)\E[\|\u_{t+1}-\u_{\tau_t}^\star\|^2]+\left(1+\gamma_t^{-1}\right)\E[\|\u_{\tau_t}^\star-\u_{\tau_{t+1}}^\star\|^2].
\end{align}
Choose $\gamma_t = \frac{\eta_t \tau_t}{4}$. Since $0\leq\eta_t \tau_t\leq 1$, we have
\begin{align*}
    &\left(1+\gamma_t\right)\left(1-\eta_t\tau_t\right)=  \left(1+\frac{\eta_t \tau_t}{4}\right)\left(1-\eta_t\tau_t\right)\leq 1-\frac{\eta_t \tau_t}{2},\\
    & 1+\gamma_t^{-1}\leq \frac{5}{\eta_t\tau_t}.
\end{align*}
 Combining these bounds with~\eqref{eq:mid_1} and~\eqref{eq:mid_2} yields
 \begin{align}\label{eq:eg_recur}
     &\E[\|\u_{t+1}-\u_{\tau_{t+1}}^\star\|^2] \notag\\
\leq\ & \left(1- \frac{\eta_t\tau_{t} }{2}\right)\E[\|\u_{t}-\u_{\tau_t}^\star\|^2]   +12\eta_t^2\sigma^2 +\frac{5}{\eta_t\tau_t}\E[\|\u_{\tau_t}^\star-\u_{\tau_{t+1}}^\star\|^2]\notag\\
\leq\ & \left(1- \frac{\eta_t\tau_{t} }{2}\right)\E[\|\u_{t}-\u_{\tau_t}^\star\|^2]   +12\eta_t^2\sigma^2 +\frac{5(\tau_t-\tau_{t+1})^2}{\eta_t\tau_t^3} \|\u^\star\|^2 ,
 \end{align}
 where the last inequality follows from Lemma~\ref{lem:per_sol}.
If we choose $\eta_t =  \eta_0(t+b)^{-\alpha} $ and $\tau_t=\tau_0(t+b)^{-\beta}$ with $b>0, \alpha>\beta>0$, and $\alpha+\beta<1$. Moreover, $\eta_0,\tau_0$ is chosen such that $\eta_0\tau_0 (t+b)^{-\alpha-\beta}\leq 1$ for all $t\geq 1$. Then it satisfies the condition in Lemma~\ref{lem:chung_new} and applying the lemma gives that 
      \[
      \E[\|\u_t-\u_{\tau_t}^\star\|^2]\leq\mO(t^{-m}) \qquad m \coloneqq \min\{\alpha-\beta, 2\left(1-\alpha-\beta\right) \}.
      \] 
By choosing $\alpha=\tfrac{3}{5}$ and $\beta=\tfrac{1}{5}$ gives the conclusion.
\item \textbf{PS-OGDA:}
By the $\tau_{t+\frac{1}{2}}$-strong monotonicity of $W$ and Young’s inequality, we have  
\begin{align*}
\langle W(\u_{t+\frac{1}{2}}), \u_{t+\frac{1}{2}}-\u_{\tau_t}^\star \rangle
\geq \tau_t\|\u_{t+\frac{1}{2}}-\u_{\tau_t}^\star\|^2
\geq \frac{\tau_t}{2}\|\u_{t}-\u_{\tau_t}^\star\|^2
- \tau_t\|\u_t-\u_{t+\frac{1}{2}}\|^2.
\end{align*}
Combining this with Lemma~\ref{lem:key_ogda}, we obtain
\begin{align*}
    &\E[\|\u_{t+1}-\u_{\tau_t}^\star\|^2]+3\eta_t^2 \E[\|W(\u_{t+\frac{1}{2}})-W(\u_{t-\frac{1}{2}})\|^2] \\
    \leq\ &  \left(1-\eta_t \tau_t\right)\E[ \|\u_{t}-\u_{\tau_t}^\star\|^2]+72\eta_{t-1}^2\eta_t^2\left(L+\tau_{t-1}\right)^2\E [\| W(\u_{t-\frac{1}{2}})-W(\u_{t-\frac{3}{2}})  \|^2]\\
    &-\left(1-24 \eta_t^2\left(L+\tau_{t-1}\right)^2-2\eta_t \tau_t\right) \E[ \| \u_t -\u_{t+\frac{1}{2}}\|^2]  \\
	&
	+12\left(\tau_{t}-\tau_{t-1}\right)^2 \eta_t^2\E[\|\u_{t-\frac{1}{2}}\|^2] + \left(144\eta_{t-1}^2\left(L+\tau_{t-1}\right)^2+6\right)\eta_t^2\sigma^2.
\end{align*}
Rearranging it gives
\begin{align}\label{eq:rec_ogd}
    &\E[\|\u_{t+1}-\u_{\tau_t}^\star\|^2]+3\eta_t^2 \E[\|W(\u_{t+\frac{1}{2}})-W(\u_{t-\frac{1}{2}})\|^2]+\frac{12\left(\tau_{t+1}-\tau_{t}\right)^2\eta_{t+1}^2}{1-\frac{\eta_{t+1} \tau_{t+1}}{2}} \E[\|\u_{t+\frac{1}{2}}\|^2] \notag\\
    \leq\ &    \left(1-\eta_t \tau_t\right)\E[ \|\u_{t}-\u_{\tau_t}^\star\|^2]+72\eta_{t-1}^2\eta_t^2\left(L+\tau_{t-1}\right)^2\E [\| W(\u_{t-\frac{1}{2}})-W(\u_{t-\frac{3}{2}})  \|^2] \notag \\
    &-\left(1-24 \eta_t^2\left(L+\tau_{t-1}\right)^2-2\eta_t \tau_t\right) \E[ \| \u_t -\u_{t+\frac{1}{2}}\|^2]  
	+12\left(\tau_{t}-\tau_{t-1}\right)^2\eta_t^2 \E[\|\u_{t-\frac{1}{2}}\|^2] \notag \\
	&+\frac{12\left(\tau_{t+1}-\tau_{t}\right)^2\eta_{t+1}^2}{1-\frac{\eta_{t+1} \tau_{t+1}}{2}}  \E[\|\u_{t+\frac{1}{2}}\|^2]+ \left(144\eta_{t-1}^2\left(L+\tau_{t-1}\right)^2+6\right)\eta_t^2\sigma^2 \notag\\
   \leq\ &    \left(1-\eta_t \tau_t+\frac{48\left(\tau_{t+1}-\tau_{t}\right)^2\eta_{t+1}^2}{1-\frac{\eta_{t+1} \tau_{t+1}}{2}} \right)\E[ \|\u_{t}-\u_{\tau_t}^\star\|^2]\notag\\
   &+72\eta_{t-1}^2\eta_t^2\left(L+\tau_{t-1}\right)^2\E [\| W(\u_{t-\frac{1}{2}})-W(\u_{t-\frac{3}{2}})  \|^2]  +12\left(\tau_{t}-\tau_{t-1}\right)^2\eta_{t }^2 \E[\|\u_{t-\frac{1}{2}}\|^2]\notag \\
    &-\left(1-24 \eta_t^2\left(L+\tau_{t-1}\right)^2-2\eta_t \tau_t-\frac{48\left(\tau_{t+1}-\tau_{t}\right)^2\eta_{t+1}^2}{1-\frac{\eta_{t+1} \tau_{t+1}}{2}} \right) \E[ \| \u_t -\u_{t+\frac{1}{2}}\|^2]  \notag \\
	&+\frac{24\left(\tau_{t+1}-\tau_{t}\right)^2\eta_{t+1}^2}{1-\frac{\eta_{t+1} \tau_{t+1}}{2}}  \|\u_{\tau_t}^\star\|^2+ \left(144\eta_{t-1}^2\left(L+\tau_{t-1}\right)^2+6\right)\eta_t^2\sigma^2 ,
\end{align}
where the second inequality is derived by using the AM-GM inequality twice:
\begin{align}\label{eq:bd_iter_t_to_t+1}
    \E[\|\u_{t+\frac{1}{2}}\|^2] \leq\ & 2\E[\|\u_{t+\frac{1}{2}}-\u_{\tau_t}^\star\|^2]+2\|\u_{\tau_t}^\star\|^2 \notag\\
    \leq\ & 4\E[\|\u_{t}-\u_{\tau_t}^\star\|^2]+ 4\E[\|\u_{t+\frac{1}{2}}-\u_{t}\|^2]+2\|\u_{\tau_t}^\star\|^2.
\end{align}
Before analyzing the recursion, we first simplify the coefficient. We assume that  $\eta_{t}\leq \min\left\{\tfrac{\tau_t}{192\left(\tau_{t}-\tau_{t-1}\right)^2}, \tfrac{1}{96 \left(L+\tau_{t-1}\right) },\tfrac{1}{5\tau_t}\right\}$ holds for all $t\geq 0$. Then, we have
\begin{align}\label{eq:mid_3}
    &\E[\|\u_{t+1}-\u_{\tau_t}^\star\|^2]+3\eta_t^2 \E[\|W(\u_{t+\frac{1}{2}})-W(\u_{t-\frac{1}{2}})\|^2]+\frac{12\left(\tau_{t+1}-\tau_{t}\right)^2\eta_{t+1}^2}{1-\frac{\eta_{t+1} \tau_{t+1}}{2}} \E[\|\u_{t+\frac{1}{2}}\|^2] \notag\\
   \leq\ &    \left(1-\frac{\eta_t \tau_t}{2}\right)\E[ \|\u_{t}-\u_{\tau_t}^\star\|^2]+\frac{3\eta_{t-1}^2}{4}\E [\| W(\u_{t-\frac{1}{2}})-W(\u_{t-\frac{3}{2}})  \|^2] \notag \\
    &+12\left(\tau_{t}-\tau_{t-1}\right)^2 \eta_{t}^2\E[\|\u_{t-\frac{1}{2}}\|^2]-\frac{1}{4} \E[ \| \u_t -\u_{t+\frac{1}{2}}\|^2]  +48\left(\tau_{t+1}-\tau_{t}\right)^2\eta_{t+1}^2 \|\u_{\tau_t}^\star\|^2+ 8\eta_t^2\sigma^2 \notag\\
    \leq\ & \left(1-\frac{\eta_t \tau_t}{2}\right)\left(\E[ \|\u_{t}-\u_{\tau_t}^\star\|^2]+ 3\eta_{t-1}^2 \E [\| W(\u_{t-\frac{1}{2}})-W(\u_{t-\frac{3}{2}})  \|^2]\right.\notag\\
    &\left.+\frac{12\left(\tau_{t}-\tau_{t-1}\right)^2\eta_{t}^2}{1-\frac{\eta_t\tau_t}{2}} \E[\|\u_{t-\frac{1}{2}}\|^2] \right) -\frac{3\eta_{t-1}^2}{4}\left(1-2\eta_t\tau_t\right)\E [\| W(\u_{t-\frac{1}{2}})-W(\u_{t-\frac{3}{2}})  \|^2]\notag \\
    & -\frac{1}{4} \E[ \| \u_t -\u_{t+\frac{1}{2}}\|^2]   +48\left(\tau_{t+1}-\tau_{t}\right)^2\eta_{t+1}^2 \|\u_{\tau_t}^\star\|^2+ 8\eta_t^2\sigma^2 \notag\\
    \leq\ &\left(1-\frac{\eta_t \tau_t}{2}\right)\left(\E[ \|\u_{t}-\u_{\tau_t}^\star\|^2]+ 3\eta_{t-1}^2 \E [\| W(\u_{t-\frac{1}{2}})-W(\u_{t-\frac{3}{2}})  \|^2]\right.\notag\\
    &\left.+\frac{12\left(\tau_{t}-\tau_{t-1}\right)^2\eta_{t }^2}{1-\frac{\eta_t\tau_t}{2}} \E[\|\u_{t-\frac{1}{2}}\|^2] \right) +48\left(\tau_{t+1}-\tau_{t}\right)^2 \eta_{t+1}^2\|\u^\star\|^2+ 8\eta_t^2\sigma^2 .
\end{align}
On the other hand, by Jensen's inequality and Lemma~\ref{lem:per_sol}, we have
 \begin{align}\label{eq:mid_4}
	\E[\|\u_{t+1}-\u_{\tau_{t+1}}^\star\|^2]\leq\ &\left(1+\frac{\eta_t \tau_t}{4}\right)\E[\|\u_{t+1}-\u_{\tau_t}^\star\|^2]+\left(1+\frac{4}{  \eta_t\tau_t}\right)\E[\|\u_{\tau_t}^\star-\u_{\tau_{t+1}}^\star\|^2] \notag\\
    \leq\ &\left(1+\frac{\eta_t \tau_t}{4}\right)\E[\|\u_{t+1}-\u_{\tau_t}^\star\|^2]+\frac{5\left(\tau_t-\tau_{t+1}\right)^2}{\eta_t\tau_t^3} \|\u^\star\|^2.
\end{align}
 Combining the bounds~\eqref{eq:mid_3} and~\eqref{eq:mid_4} yields
\begin{align}\label{eq:og_recur}
    &\E[\|\u_{t+1}-\u_{\tau_{t+1}}^\star\|^2]+3\eta_t^2 \E[\|W(\u_{t+\frac{1}{2}})-W(\u_{t-\frac{1}{2}})\|^2]+\frac{12\left(\tau_{t+1}-\tau_{t}\right)^2\eta_{t+1}^2}{1-\frac{\eta_{t+1} \tau_{t+1}}{2}} \E[\|\u_{t+\frac{1}{2}}\|^2] \notag\\
    \leq\ &  \left(1- \frac{\eta_t\tau_{t} }{4}\right) \left(\E[ \|\u_{t}-\u_{\tau_t}^\star\|^2]+ 3\eta_{t-1}^2 \E [\| W(\u_{t-\frac{1}{2}})-W(\u_{t-\frac{3}{2}})  \|^2]\right.\notag\\
    &\left.+\frac{12\left(\tau_{t}-\tau_{t-1}\right)^2\eta_{t }^2}{1-\frac{\eta_t\tau_t}{2}} \E[\|\u_{t-\frac{1}{2}}\|^2] \right)  +\frac{5(\tau_t-\tau_{t+1})^2}{\eta_t\tau_t^3} \|\u^\star\|^2  \notag\\
    &+ 96\left(\tau_{t+1}-\tau_{t}\right)^2\eta_{t+1}^2 \|\u^\star\|^2+ 16\eta_t^2\sigma^2.
\end{align}
If we choose $\eta_t =  \eta_0(t+b)^{-\alpha} $ and $\tau_t=\tau_0(t+b)^{-\beta}$ with $b>0, \alpha>\beta>0$, and $\alpha+\beta<1$. Moreover, $\eta_0,\tau_0$ is chosen such that $\eta_0\tau_0 (t+b)^{-\alpha-\beta}\leq 1$ for all $t\geq 1$. Define the Lyapunov function
\[
\Phi_t \coloneqq \E[ \|\u_{t}-\u_{\tau_t}^\star\|^2]+ 3\eta_{t-1}^2 \E [\| W(\u_{t-\frac{1}{2}})-W(\u_{t-\frac{3}{2}})  \|^2]+\frac{12\left(\tau_{t}-\tau_{t-1}\right)^2\eta_{t}^2}{1-\frac{\eta_t\tau_t}{2}} \E[\|\u_{t-\frac{1}{2}}\|^2].
\]
Then we have 
\[
\Phi_{t+1}\leq \left(1-\frac{\eta_t \tau_t}{4}\right)\Phi_t+\frac{5(\tau_t-\tau_{t+1})^2}{\eta_t\tau_t^3} \|\u^\star\|^2 + 96\left(\tau_{t+1}-\tau_{t}\right)^2\eta_{t+1}^2 \|\u^\star\|^2+ 16\eta_t^2\sigma^2.
\]
Then it satisfies the condition in Lemma~\ref{lem:chung_new} and applying the lemma gives that 
      \[
      \E[\|\u_t-\u_{\tau_t}^\star\|^2]\leq \Phi_t\leq\mO(t^{-m}) \qquad m \coloneqq \min\{\alpha-\beta, 2\left(1-\alpha-\beta\right) \}. 
      \] 
By choosing $\alpha=\tfrac{3}{5}$ and $\beta=\tfrac{1}{5}$, we finish the proof.
\end{enumerate} 
\end{proof} 
 
\subsection{From Distance to Stationarity}\label{sec:dist_stat}
 While the above result establishes convergence to the perturbed solution $\u_\tau^\star$, it does not directly characterize optimality for the original problem. 
To bridge this gap, we relate the distance $\E[\|\u - \u_{\tau}^\star\|^2]$ to the restricted primal--dual gap or gradient norm.

A key step is to ensure that the restricted primal-dual gap is well-defined. Since the feasible set may be unbounded, we cannot directly work with the primal--dual gap. Instead, we show that the iterates remain uniformly bounded in mean square, which allows us to restrict attention, in expectation, to a compact region containing a saddle point.
\begin{lemma}[Uniform mean-square boundedness of iterates] \label{lem:bd_iter}
Let $\{\u_t\}_{t=1}^T$ and $\{\u_{t-\frac{1}{2}}\}_{t=1}^T$ be the sequences generated by either PS-EG or PS-OGDA. 
Suppose the assumptions of either Theorem~\ref{thm:pert_rate} or Theorem~\ref{thm:pert_any} hold. Then there exists a positive constant $C$ which depends on $\|\u_0-\u^\star\|^2$, $\|\u_0\|^2$, $\sigma$, and parameters $\eta_0, \tau_0$, such that 
\[
\sup_{t \geq 0} \E[\|\u_t - \u^\star\|^2] \leq C. 
\]
\end{lemma}
Lemma~\ref{lem:bd_iter} ensures that the iterates remain uniformly bounded in mean square throughout the optimization process. Based on this bound, we define the compact set 
$\mathcal{B}_{\x} \coloneqq \X\cap \mathcal{B}(\x^\star, \sqrt{C}), \mathcal{B}_{\y} \coloneqq  \Y\cap \mathcal{B}(\y^\star, \sqrt{C})$, which contain the saddle point $\u^\star$. We are now ready to establish the key conversion result, which relates the distance to the perturbed saddle point to the restricted primal--dual gap of the original problem. 

\begin{lemma}[Distance-to-gap conversion]\label{lem:relation}
Let Assumptions~\ref{ass:cc}, \ref{ass:l_smooth}, \ref{ass:exist}, and~\ref{ass:stoc} hold.  
Suppose $\u = (\x, \y) \in \U$ satisfies $\E[\|\u - \u_\tau^\star\|^2] \leq \epsilon$. Then there exists a finite constant $M_R>0$, depending only on $F$, $\u^\star$, and a compact neighborhood
$\mathcal{B}_{\x}\times\mathcal{B}_{\y}$, where
$\mathcal{B}_{\x}$ and $\mathcal{B}_{\y}$ are centered at
$\x^\star$ and $\y^\star$, respectively, such that
\[
\E[\mathcal{G}^R(\u)]
\leq M_R\sqrt{\epsilon} + \frac{L\epsilon }{2} + \frac{\tau}{2}\left(2\diam(\mathcal{B}_{\x})^2+2\diam(\mathcal{B}_{\y})^2+4\|\u^\star\|^2\right).
\]
In particular,
\[
\E[\mathcal{G}^R(\u)]\le \mathcal{O}(\sqrt{\epsilon}+\tau).
\]
\end{lemma}

\begin{proof}
We first define the optimal solution of $F$ within the restricted sets:
\[
 \y(\x) = \argmax_{\y' \in \mathcal{B}_{\y}} F(\x,\y') \quad \text{and} \quad \x(\y) = \argmin_{\x' \in \mathcal{B}_{\x}} F(\x',\y).
\]
Thus, the restricted gap evaluates to exact equality as $\mathcal{G}^R(\u) = F(\x, \y(\x)) - F(\x(\y), \y)$. To provide an upper bound for $\mathcal{G}^R$, we decompose it into four parts and then bound it separately.
\begin{align*}
    \mathcal{G}^R(\u) =\ & \underbrace{F(\x, \y(\x)) - F(\x_{\tau}^\star, \y(\x))}_
    {\textrm{optimization\, error}}  +  \underbrace{F(\x_{\tau}^\star, \y(\x)) - F(\x_{\tau}^\star, \y_{\tau}^\star) + F(\x_{\tau}^\star, \y_{\tau}^\star) - F(\x(\y), \y_{\tau}^\star)}_{\textrm{approximation\, error}}\\
    &+ \underbrace{F(\x(\y), \y_{\tau}^\star) - F(\x(\y), \y)}_{\textrm{optimization\, error}}.
\end{align*}

To bound the optimization error, by $L$-smoothness of $F$, we have
 \begin{align*} F(\x, \y(\x)) - F(\x_{\tau}^\star, \y(\x)) &\leq \langle \nabla_{\x} F(\x_{\tau}^\star, \y(\x)), \x - \x_{\tau}^\star \rangle + \frac{L}{2}\|\x - \x_{\tau}^\star\|^2 \\ F(\x(\y), \y_{\tau}^\star) - F(\x(\y), \y) &\leq \langle -\nabla_{\y} F(\x(\y), \y_{\tau}^\star), \y - \y_{\tau}^\star \rangle + \frac{L}{2}\|\y - \y_{\tau}^\star\|^2 \end{align*}
Let $V_{\mathcal{B}}(\u) = [\nabla_{\x} F(\x_{\tau}^\star, \y(\x)); -\nabla_{\y} F(\x(\y), \y_{\tau}^\star)]$. Then we can bound the optimization error:
\begin{align*}
    \text{optimization error}\leq \langle V_{\mathcal{B}}(\u), \u - \u_{\tau}^\star \rangle + \frac{L}{2}\|\u - \u_{\tau}^\star\|^2 . 
\end{align*}

Next, we bound the approximation error. Noting that $\u_{\tau}^\star$ is the saddle point of the perturbed function $F_{\tau}(\x,\y)=F(\x,\y)+\tfrac{\tau}{2}\|\x\|^2-\tfrac{\tau}{2}\|\y\|^2$, it satisfies $F_{\tau}(\x_{\tau}^\star, \y) \leq F_{\tau}(\x_{\tau}^\star, \y_{\tau}^\star) \leq F_{\tau}(\x, \y_{\tau}^\star)$ for all $\x\in \X,\y\in \Y$. Therefore, we have 
\begin{align*}
    F (\x_{\tau}^\star, \y(\x)) - F (\x_{\tau}^\star, \y_{\tau}^\star) = \ & \left(F_{\tau}(\x_{\tau}^\star, \y(\x))+\frac{\tau}{2}\|\y(\x)\|^2\right) - \left(F_{\tau}(\x_{\tau}^\star, \y_{\tau}^\star) +\frac{\tau}{2}\|\y_{\tau}^\star\|^2\right)\\
    \leq\ & \frac{\tau}{2}\|\y(\x)\|^2 \\ 
    F(\x_{\tau}^\star, \y_{\tau}^\star) - F(\x(\y), \y_{\tau}^\star)=\ &\left(F_{\tau}(\x_{\tau}^\star, \y_{\tau}^\star)-\frac{\tau}{2}\|\x_{\tau}^\star\|^2\right) - \left(F_{\tau}(\x(\y), \y_{\tau}^\star)-\frac{\tau}{2}\|\x(\y)\|^2 \right)\\  \leq\ & \frac{\tau}{2}\|\x(\y)\|^2.
\end{align*}
Summing these, together with the bound for optimization error, we have
\begin{align*}
   \E[\mathcal{G}^R(\u)]\leq\ & \E[\langle V_{\mathcal{B}}(\u), \u - \u_{\tau}^\star \rangle] +\frac{L}{2}\E[\|\u - \u_{\tau}^\star\|^2]+\frac{\tau}{2}\E[\left( \|\y(\x)\|^2+ \|\x(\y)\|^2\right)] \\
   \leq\ & \sqrt{\E[\| V_{\mathcal{B}}(\u)\|^2]}\sqrt{\E[\|\u - \u_{\tau}^\star\|^2]}+\frac{L}{2}\E[\|\u - \u_{\tau}^\star\|^2]+\frac{\tau}{2}\E[\left( \|\y(\x)\|^2+ \|\x(\y)\|^2\right)],
\end{align*}
where the second inequality follows from the Cauchy-Schwarz inequality. Since $\mathcal{B}_{\x}\times\mathcal{B}_{\y}$ is compact and $F$ is continuously differentiable, there exists a constant $M_R>0$ such that $\E[\| V_{\mathcal{B}}(\u)\|^2]\leq M_R^2$. Therefore, we get  
\begin{align*}
     \E[\mathcal{G}^R(\u)] \leq M_R\sqrt{\epsilon} + \frac{L\epsilon }{2} + \frac{\tau}{2}\left(2\diam(\mathcal{B}_{\x})^2+2\diam(\mathcal{B}_{\y})^2+4\|\u^\star\|^2\right) \leq \mO(\sqrt{\epsilon}+\tau).
\end{align*}
This completes the proof.
\end{proof} 
 In the unconstrained setting where $\U=\R^n\times\R^d$, the distance estimate can be further translated into a convergence guarantee for the norm of the gradient operator. The following lemma establishes this relationship. 

\begin{lemma}[Distance-to-stationarity conversion] \label{lem:rela_sta}Let Assumptions~\ref{ass:cc}, \ref{ass:l_smooth}, \ref{ass:exist}, and~\ref{ass:stoc} hold.  
Suppose $\u = (\x, \y) \in \R^n\times\R^d$ satisfies $\E[\|\u - \u_\tau^\star\|^2] \leq \epsilon$. Then we have 
\[
\E[\|V(\u)\|]\le \mathcal{O}(\sqrt{\epsilon}+\tau).
\]
\end{lemma}
\begin{proof}  
By the optimality condition of $\u_\tau^\star$, we have
\[
V(\u_\tau^\star) = -\tau \u_\tau^\star.
\]
Using Lipschitz continuity of $V$ (Assumption~\ref{ass:l_smooth}), we obtain
\begin{align*}
\|V(\u)\|^2 
&\leq 2\|V(\u)-V(\u_\tau^\star)\|^2 + 2\|V(\u_\tau^\star)\|^2  \leq 4L^2\|\u-\u_\tau^\star\|^2 + 2\tau^2\|\u_\tau^\star\|^2.
\end{align*}
Taking expectation and applying Jensen's inequality, we get
\begin{align*}
    \E[\|V(\u)\|]
\leq \sqrt{\E[\|V(\u)\|^2]}
\leq\ & 2L\sqrt{\E[\|\u-\u_\tau^\star\|^2]} + \sqrt{2}\tau\|\u_\tau^\star\|\\
\leq\ & 2L\sqrt{\E[\|\u-\u_\tau^\star\|^2]} + \sqrt{2}\tau\|\u^\star\|,
\end{align*}
where the last inequality follows from Lemma~\ref{lem:per_sol}.
Therefore, if $\E[\|\u-\u_\tau^\star\|^2]\leq \epsilon$, then we have
\[
\E[\|V(\u)\|] \leq \mathcal{O}(\sqrt{\epsilon}+\tau).
\]
This completes the proof.
\end{proof}

\subsection{Proofs of the Theorems~\ref{thm:rate_L2} and~\ref{thm:rate_any}}
\label{sec:proved_all}
We are now ready to complete the proofs of the main convergence results by combining the distance estimates established in this section with the conversion lemmas developed above.

\begin{proof}[Proof of Theorem~\ref{thm:rate_L2}]
We first establish the convergence rate for the restricted primal--dual gap.

\paragraph{Polynomial stepsize.}
By Theorem~\ref{thm:pert_rate}, we have
\[
\epsilon_T
\coloneqq 
\E \left[\|\u_T-\u_\tau^\star\|^2\right]
=
\mathcal{O} \left(T^{-\alpha+\delta}\right),
\]
where $\tau=T^{-\delta}$.
Applying Lemma~\ref{lem:relation} yields
\[
\E[\mathcal G^R(\u_T)]
=
\mathcal{O} \left(
\sqrt{\epsilon_T}
+\tau
\right)
=
\mathcal{O} \left(
T^{-(\alpha-\delta)/2}
+
T^{-\delta}
\right).
\]
Choosing $\delta=\frac{\alpha}{3}$ 
balances the two terms and gives
\[
\E[\mathcal G^R(\u_T)]
=
\mathcal{O} \left(T^{-\alpha/3}\right).
\]
Since $\alpha<3/4$, taking $\alpha=3/4-3\varepsilon$ for any $\varepsilon>0$
yields
\[
\E[\mathcal G^R(\u_T)]
=
\mathcal{O} \left(T^{-1/4+\varepsilon}\right).
\]

\paragraph{Regularization-aware harmonic stepsize.}
By Theorem~\ref{thm:pert_rate},
\[
\epsilon_T
=
\mathcal{O}(T^{-1/2}).
\]
With $\tau=T^{-1/4}$, Lemma~\ref{lem:relation} implies
\[
\E[\mathcal G^R(\u_T)]
=
\mathcal{O} \left(
T^{-1/4}
\right).
\]

Finally, in the unconstrained setting, the gradient norm estimates follow immediately from Lemma~\ref{lem:rela_sta} by substituting the above distance bounds. This completes the proof.
\end{proof}
\begin{proof}[Proof of Theorem~\ref{thm:rate_any}]
By Theorem~\ref{thm:pert_any}, the anytime algorithm satisfies
\[
\epsilon_t
\coloneqq 
\E \left[\|\u_t-\u_{\tau_t}^\star\|^2\right]
=
\mathcal{O}(t^{-2/5}),
\]
where $\tau_t
=
\mathcal{O}(t^{-1/5}).$ 
Applying Lemma~\ref{lem:relation}, we obtain
\[
\E[\mathcal G^R(\u_t)]
=
\mathcal{O} \left(
\sqrt{\epsilon_t}
+
\tau_t
\right)
=
\mathcal{O}(t^{-1/5}).
\]
Moreover, in the unconstrained setting, Lemma~\ref{lem:rela_sta} further gives
\[
\E[\|V(\u_t)\|]
=
\mathcal{O}(t^{-1/5}).
\]
This completes the proof.
\end{proof}
\subsection{A Sharper Unconstrained Analysis}
 \label{sec:improved}

In this subsection, we prove Theorem~\ref{thm:rate_any_uncon_eg}. The proof is specific to the unconstrained setting and does not proceed through the moving perturbed saddle point $\u_{\tau_t}^\star$. Instead, we compare the stochastic trajectory with a deterministic reference trajectory generated by the same time-varying perturbed EG maps.

We first record a deterministic contraction property of the unconstrained EG map.

\begin{lemma}[Contraction of the unconstrained EG map]
\label{lem:eg_map_contraction}
Let $W:\mathbb R^{n+d}\to\mathbb R^{n+d}$ be $\mu$-strongly monotone and $M$-Lipschitz continuous, where $0<\mu\le M$. For $\eta>0$, define
\[
\mathcal T_{\eta}(\u)
\coloneqq
\u-\eta W(\u-\eta W(\u)).
\]
If $0<\eta\le \frac{\mu}{16M^2},$ 
then, for all $\u,\v\in\mathbb R^{n+d}$,
\[
\|\mathcal T_{\eta}(\u)-\mathcal T_{\eta}(\v)\|^2
\le
(1-\eta\mu)\|\u-\v\|^2.
\]
Moreover, for all $\u\in\mathbb R^{n+d}$,
\[
\|W(\mathcal T_{\eta}(\u))\|
\le
\left(1-\frac{\eta\mu}{2}\right)\|W(\u)\|.
\]
\end{lemma}

\begin{proof} Fix $\u\in\mathbb R^{n+d}$ and define
\[
\widetilde{\u}\coloneqq \u-\eta W(\u),
\qquad
\u^+\coloneqq \mathcal T_{\eta}(\u)
=
\u-\eta W(\widetilde{\u}).
\]
Firstly, by the $M$-Lipschitz continuous property of $W$, we have 
\begin{align*}
&\|\mathcal T_{\eta}(\u)-\mathcal T_{\eta}(\v)\|^2\\
=\ & \|\u-\v\|^2+\eta^2 \| W(\widetilde{\u}) - W(\widetilde{\v})\|^2-2\eta \langle \u-\v, W(\widetilde{\u}) - W(\widetilde{\v}) \rangle\\
\leq\ &\|\u-\v\|^2+\eta^2 M^2\left(1+\eta M\right)^2\|\u-\v\|^2-\underbrace{2\eta \langle \u-\v, W(\widetilde{\u}) - W(\widetilde{\v}) \rangle}_{\O1}.
\end{align*}
Then, we focus on the term \O1, using the $\mu$-strongly monotonicity of $W$, we have 
\begin{align*}
   \O1
   =\ &2\eta \langle \widetilde{\u}-\widetilde{\v},W(\widetilde{\u}) - W(\widetilde{\v})\rangle + 2\eta^2 \langle W(\u)-W(\v),W(\widetilde{\u}) - W(\widetilde{\v})\rangle\\
   \geq\ & 2\eta \mu \|\widetilde{\u}-\widetilde{\v}\|^2  -  2\eta^2 M\left(1+\eta M\right) \|\u-\v\|^2\\
   \geq\ & 2\eta \mu \left(1-\eta M\right)^2 \|\u-\v\|^2-  2\eta^2 M\left(1+\eta M\right) \|\u-\v\|^2\\
   =\ & 2\eta \left(  \mu \left(1-\eta M\right)^2- \eta  M\left(1+\eta M\right)\right) \|\u-\v\|^2\geq \frac{13 \mu\eta}{8} \|\u-\v\|^2.
\end{align*}
where the first two inequalities follow from $M$-Lipschitz continuous property fo $W$. Combining these two bounds, we obtain that 
\begin{align*}
    \|\mathcal T_{\eta}(\u)-\mathcal T_{\eta}(\v)\|^2 
    \leq\ & \left(1+\eta^2 M^2\left(1+\eta M\right)^2 -\frac{13 \mu\eta}{8})\right) \|\u-\v\|^2\\
    \leq\ & \left(
1-\frac{13}{8}\eta\mu
+
\frac{289}{256}\eta^2M^2
\right)\|\u-\v\|^2 \\
 \le\ &
(1-\eta\mu)\|\u-\v\|^2,
\end{align*}
where the last inequality again uses $\eta M^2\le \frac{\mu}{16}$.
  
We next prove the residual contraction.  
By strong monotonicity and Lipschitz continuity,
\begin{align*}
\|W(\widetilde{\u})\|^2
=\ &
\|W(\u)\|^2
-
2\langle W(\u),W(\u)-W(\widetilde{\u})\rangle
+
\|W(\u)-W(\widetilde{\u})\|^2  \\
=\ & \|W(\u)\|^2
-
 \frac{2}{\eta}\langle \u-\widetilde{\u} ,W(\u)-W(\widetilde{\u})\rangle
+
\|W(\u)-W(\widetilde{\u})\|^2  \\
\leq\ & 
\|W(\u)\|^2
-
2\eta\mu\|W(\u)\|^2
+
\eta^2M^2\|W(\u)\|^2  \\
\leq\ & 
\left(1-\frac{31}{16}\eta\mu\right)\|W(\u)\|^2.
\end{align*}
Since $\eta\mu\le \frac{1}{16}$, this implies
\[
\|W(\widetilde{\u})\|
\le
\left(1-\frac{31}{32}\eta\mu\right)\|W(\u)\|.
\]
Furthermore,
\begin{align*}
\|W(\u^+)-W(\widetilde{\u})\|
\leq   
M\|\u^+-\widetilde{\u}\|  =
\eta M\|W(\u)-W(\widetilde{\u})\|   \leq
\eta^2M^2\|W(\u)\|  \leq
\frac{\eta\mu}{16}\|W(\u)\|.
\end{align*}
Combining the last two displays gives
\[
\|W(\u^+)\|
\le
\left(
1-\frac{31}{32}\eta\mu+\frac{1}{16}\eta\mu
\right)\|W(\u)\|
\le
\left(1-\frac{\eta\mu}{2}\right)\|W(\u)\|.
\]
This proves the claim.
\end{proof}

We next introduce the deterministic reference trajectory. Let $\bar{\u}_0=\u_0$ and define
\[
\bar{\u}_{t+\frac12}
=
\bar{\u}_t-\eta_t W(\bar{\u}_t),
\qquad
\bar{\u}_{t+1}
=
\bar{\u}_t-\eta_t W(\bar{\u}_{t+\frac12}).
\]
with $\tau_t=\tau_{t+\frac{1}{2}}$.
\begin{lemma}[Bounded reference trajectory and residual decay]
\label{lem:det_ref_uncon}
Under the assumptions of Theorem~\ref{thm:rate_any_uncon_eg}, there exists a finite constant $B>0$ such that
\[
\sup_{t\ge0}\|\bar{\u}_t\|\le B.
\]
Moreover,
\[
\|W(\bar{\u}_t)\|=\mO(t^{-1/4}),
\qquad
\|V(\bar{\u}_t)\|^2=\mO(t^{-1/2}).
\]
\end{lemma}

\begin{proof}
We first prove boundedness. Since $\U=\mathbb R^{n+d}$, Assumption~\ref{ass:exist} implies that there exists $\u^\star$ such that $V(\u^\star)=0$. Applying Lemma~\ref{lem:proj_key} with $\U=\mathbb R^{n+d}$ to the deterministic EG update gives
\begin{align*}
\|\bar{\u}_{t+1}-\u^\star\|^2
\leq\ &
\|\bar{\u}_{t}-\u^\star\|^2
-
2\eta_t
\langle
W(\bar{\u}_{t+\frac12}),
\bar{\u}_{t+\frac12}-\u^\star
\rangle  +
\eta_t^2
\|W(\bar{\u}_{t})-W(\bar{\u}_{t+\frac12})\|^2
-
\|\bar{\u}_{t+\frac12}-\bar{\u}_{t}\|^2  \\
\leq\ &
\|\bar{\u}_{t}-\u^\star\|^2
-
2\eta_t
\langle
W(\bar{\u}_{t+\frac12}),
\bar{\u}_{t+\frac12}-\u^\star
\rangle -
\left(1-\eta_t^2M_t^2\right)
\|\bar{\u}_{t+\frac12}-\bar{\u}_{t}\|^2.
\end{align*}
By monotonicity of $V$ and $V(\u^\star)=0$, we have
\[
\langle
V(\bar{\u}_{t+\frac12}),
\bar{\u}_{t+\frac12}-\u^\star
\rangle
\ge0.
\]
Moreover,
\[
\langle
\bar{\u}_{t+\frac12},
\bar{\u}_{t+\frac12}-\u^\star
\rangle
=
\frac12\|\bar{\u}_{t+\frac12}-\u^\star\|^2
+
\frac12\|\bar{\u}_{t+\frac12}\|^2
-
\frac12\|\u^\star\|^2
\ge
\frac12\|\bar{\u}_{t+\frac12}-\u^\star\|^2
-
\frac12\|\u^\star\|^2.
\]
Therefore,
\[
\langle
W(\bar{\u}_{t+\frac12}),
\bar{\u}_{t+\frac12}-\u^\star
\rangle
\ge
\frac{\tau_t}{2}
\|\bar{\u}_{t+\frac12}-\u^\star\|^2
-
\frac{\tau_t}{2}\|\u^\star\|^2.
\]
Using
\[
\|\bar{\u}_{t+\frac12}-\u^\star\|^2
\ge
\frac12\|\bar{\u}_{t}-\u^\star\|^2
-
\|\bar{\u}_{t+\frac12}-\bar{\u}_{t}\|^2,
\]
we obtain
\begin{align*}
\|\bar{\u}_{t+1}-\u^\star\|^2
\leq\ &
\left(1-\frac{\eta_t\tau_t}{2}\right)
\|\bar{\u}_{t}-\u^\star\|^2
+
\eta_t\tau_t\|\u^\star\|^2  -
\left(1-\eta_t\tau_t-\eta_t^2M_t^2\right)
\|\bar{\u}_{t+\frac12}-\bar{\u}_{t}\|^2.
\end{align*}
By the stepsize condition in Theorem~\ref{thm:rate_any_uncon_eg}, $\eta_t^2M_t^2
\le
\frac{\eta_t\tau_t}{16}, 
\eta_t\tau_t\le \frac12,$ we have $1-\eta_t\tau_t-\eta_t^2M_t^2\ge0.$ 
Thus, with
\[
a_t\coloneqq \|\bar{\u}_t-\u^\star\|^2,
\]
we have
\[
a_{t+1}
\le
\left(1-\frac{\eta_t\tau_t}{2}\right)a_t
+
\eta_t\tau_t\|\u^\star\|^2.
\]
It follows by mathematical induction that
\[
a_t
\le
A
\coloneqq
\max\{a_0,2\|\u^\star\|^2\},
\qquad \forall t\ge0.
\]
Therefore,
\[
\|\bar{\u}_t\|
\le
\|\u^\star\|+\sqrt{A}
\eqqcolon B,
\qquad \forall t\ge0.
\]
 
For fixed $t$, the operator $W$ is $\tau_t$-strongly monotone and $M_t\left(M_t\coloneqq 2L+\tau_t\right)$-Lipschitz continuous. Applying Lemma~\ref{lem:eg_map_contraction} with $W(\cdot)= V(\cdot)+\tau_t \cdot $ and using the stepsize condition and $\tau_{t+1}\le\tau_t$, we obtain
\begin{align*}
\|W(\bar{\u}_{t+1})\| \leq\ & \left(1-\frac{\eta_t\tau_t}{2}\right)\|W(\bar{\u}_{t })\|+ 
\left(\tau_t-\tau_{t+1}\right)\|\bar{\u}_{t+1}\|\\
\le\ &
\left(1-\frac{\eta_t\tau_t}{2}\right)\|W(\bar{\u}_{t })\|   
+
B(\tau_t-\tau_{t+1}).
\end{align*}
By the schedules, $\eta_t\tau_t
=
\eta_0\tau_0(t+b)^{-1},$
and
\[
\tau_t-\tau_{t+1}
=
\tau_0\left((t+b)^{-1/4}-(t+b+1)^{-1/4}\right)
\le
\frac{\tau_0}{4}(t+b)^{-5/4}.
\] 
Then
\[
\|W(\bar{\u}_{t+1})\|  
\le
\left(1-\frac{ \eta_0\tau_0}{2(t+b)}\right)\|W_{t}(\bar{\u}_{t})\|  
+
\frac{\tau_0 B}{4}(t+b)^{-5/4}.
\]
Since $\eta_0\tau_0>1$, we can apply Lemma~\ref{lem:chung_harmonic_general} with $r=\tfrac14$ yields
\[
\|W_{t}(\bar{\u}_{t})\| \leq\mO(t^{-1/4}).
\]
Finally, by the definition of $V$, we have 
\[
\|V(\bar{\u}_t)\|
\le
\|W(\bar{\u}_t)\|+\tau_t\|\bar{\u}_t\|
\le 
\mO(t^{-1/4}),
\]
which  completes the proof.
\end{proof}

We now show that the stochastic trajectory tracks the deterministic reference trajectory.

\begin{lemma}[Stochastic tracking]
\label{lem:stoch_track_uncon}
Under the assumptions of Theorem~\ref{thm:rate_any_uncon_eg}, we have
\[
\E[\|\u_t-\bar{\u}_t\|^2]
=
\mO(t^{-1/2}).
\]
\end{lemma}

\begin{proof}
For fixed $t$, define the deterministic EG map associated with $W$:
\[
\mathcal T_t(\u)
\coloneqq
\u-\eta_t W(\u-\eta_t W(\u)).
\]
Let
\[
\e_t\coloneqq \u_t-\bar{\u}_t,
\qquad
\p_t\coloneqq \u_t-\eta_t W(\u_t).
\]
The stochastic intermediate point satisfies
\[
\u_{t+\frac12}
=
\p_t-\eta_t\bm{\zeta}_t.
\]
Thus,
\begin{align*}
\u_{t+1}
&=
\u_t-\eta_t\left(W(\u_{t+\frac12})+\bm{\zeta}_{t+\frac12}\right)  =
\mathcal T_t(\u_t)
+
\Delta_t
-
\eta_t\bm{\zeta}_{t+\frac12},
\end{align*}
where
\[
\Delta_t
\coloneqq
-\eta_t\left(W(\u_{t+\frac12})-W(\p_t)\right).
\]
By Lipschitz continuity of $W$,
\[
\|\Delta_t\|
\le
\eta_tM_t\|\u_{t+\frac12}-\p_t\|
=
\eta_t^2M_t\|\bm{\zeta}_t\|.
\]
Since $\bar{\u}_{t+1}=\mathcal T_t(\bar{\u}_t)$, we have
\[
\e_{t+1}
=
\mathcal T_t(\u_t)-\mathcal T_t(\bar{\u}_t)
+
\Delta_t
-
\eta_t\bm{\zeta}_{t+\frac12}.
\]
Conditioning on $\mathcal F_{t+\frac12}$ and using Assumption~\ref{ass:stoc}, we obtain
\[
\E\!\left[\|\e_{t+1}\|^2\mid\mathcal F_{t+\frac12}\right]
\le
\|\mathcal T_t(\u_t)-\mathcal T_t(\bar{\u}_t)+\Delta_t\|^2
+
\eta_t^2\sigma^2.
\] 
By Young's inequality, we get
\begin{align*}
\|\mathcal T_t(\u_t)-\mathcal T_t(\bar{\u}_t)+\Delta_t\|^2
\leq\ &
(1+ \frac{\eta_t\tau_t}{4})
\|\mathcal T_t(\u_t)-\mathcal T_t(\bar{\u}_t)\|^2  +
(1+ \frac{4}{\eta_t\tau_t} )
\|\Delta_t\|^2,
\end{align*}
By Lemma~\ref{lem:eg_map_contraction}, we have
\[
\|\mathcal T_t(\u_t)-\mathcal T_t(\bar{\u}_t)\|^2
\le
(1-\eta_t\tau_t)\|\e_t\|^2.
\]
Since $\eta_t\tau_t\le\frac12$, we have
\[
(1+\rho_t)(1-\eta_t\tau_t)
\le
1-\frac{\eta_t\tau_t}{2},
\qquad
1+\rho_t^{-1}
\le
\frac{5}{\eta_t\tau_t}.
\]
Moreover, $\E[\|\Delta_t\|^2]
\le
\eta_t^4M_t^2\sigma^2.$ 
Taking total expectation, with $r_t\coloneqq \E[\|\e_t\|^2],$ 
we get
\[
r_{t+1}
\le
\left(1-\frac{\eta_t\tau_t}{2}\right)r_t
+
\eta_t^2\sigma^2
+
\frac{5\eta_t^3M_t^2}{\tau_t}\sigma^2.
\]
Let $\bar M\coloneqq L+\tau_0b^{-1/4}.$ 
Then $M_t\le\bar M$ for all $t\ge0$. Using
\[
\eta_t^2=\eta_0^2(t+b)^{-3/2},
\qquad
\frac{\eta_t^3}{\tau_t}
=
\frac{\eta_0^3}{\tau_0}(t+b)^{-2},
\]
and $\eta_t\tau_t=\eta_0\tau_0(t+b)^{-1}$, we obtain
\[
r_{t+1}
\le
\left(1-\frac{\theta}{t+b}\right)r_t
+
C_1(t+b)^{-3/2}
+
C_2(t+b)^{-2},
\qquad
\theta\coloneqq \frac{\eta_0\tau_0}{2}.
\]
Since $(t+b)^{-2}\le (t+b)^{-3/2}$ for $b\ge1$, this implies
\[
r_{t+1}
\le
\left(1-\frac{\theta}{t+b}\right)r_t
+
C(t+b)^{-3/2}.
\]
Because $\eta_0\tau_0>1$, we have $\theta>\frac12$. Applying Lemma~\ref{lem:chung_harmonic_general} with $r=\tfrac12$ gives
\[
r_t=\mO(t^{-1/2}).
\]
This completes the proof.
\end{proof}

\begin{proof}[Proof of Theorem~\ref{thm:rate_any_uncon_eg}]
By Lipschitz continuity of $V$, we have
\begin{align*}
\E[\|V(\u_t)\|^2]
&\le
2\E[\|V(\u_t)-V(\bar{\u}_t)\|^2]
+
2\|V(\bar{\u}_t)\|^2  \\
&\le
2L^2\E[\|\u_t-\bar{\u}_t\|^2]
+
2\|V(\bar{\u}_t)\|^2.
\end{align*}
By Lemma~\ref{lem:stoch_track_uncon},
\[
\E[\|\u_t-\bar{\u}_t\|^2]=\mO(t^{-1/2}),
\]
and by Lemma~\ref{lem:det_ref_uncon},
\[
\|V(\bar{\u}_t)\|^2=\mO(t^{-1/2}).
\]
Therefore,
\[
\E[\|V(\u_t)\|^2]
=
\mO(t^{-1/2}).
\]
Finally, Jensen's inequality gives
\[
\E[\|V(\u_t)\|]
\le
\sqrt{\E[\|V(\u_t)\|^2]}
=
\mO(t^{-1/4}).
\]
This completes the proof.
\end{proof}
\section{Conclusion}\label{sec:6}

In this paper, we studied last-iterate convergence of stochastic first-order methods for smooth convex--concave minimax optimization under the standard bounded-variance stochastic oracle. We proposed a perturbation framework for stochastic EG and stochastic OGDA and established non-asymptotic last-iterate convergence guarantees under standard smoothness assumptions.

For the horizon-dependent setting, we showed that both PS-EG and PS-OGDA achieve an $\mathcal{O}(T^{-1/4})$ last-iterate convergence rate for the restricted primal--dual gap. We further developed an anytime variant based on diminishing perturbations and diminishing stepsizes. For general closed convex feasible sets, both methods achieve an $\mathcal{O}(T^{-1/5})$ last-iterate convergence rate for the restricted primal--dual gap. In the unconstrained setting, we further established a sharper $\mathcal{O}(T^{-1/4})$ anytime convergence rate in terms of the gradient norm for PS-EG.

Several interesting questions remain open.

First, although PS-EG achieves an $\mathcal{O}(T^{-1/4})$ anytime gradient-norm convergence rate in the unconstrained setting, PS-OGDA currently admits only an $\mathcal{O}(T^{-1/5})$ guarantee. Whether the sharper $\mathcal{O}(T^{-1/4})$ rate can also be established for PS-OGDA remains an interesting open question.

Second, in the unconstrained setting, faster rates of $\mathcal{O}(T^{-1/2})$ (in terms of the gradient norm) are known for multi-loop algorithms~\cite{chen2024near}, where convergence is established for the final outer iterate. Extending such ideas to constrained minimax optimization and obtaining faster last-iterate convergence guarantees is another promising direction.

Finally, our horizon-dependent rate remains slower than the best known $\mathcal{O}(T^{-1/3})$ last-iterate guarantee in the unconstrained setting obtained under stronger stochastic oracle models~\cite{cai2022stochastic}. It remains an interesting question whether such a rate can be achieved under the standard bounded-variance stochastic oracle.
\bibliographystyle{abbrvnat}
\bibliography{references}
\section*{Appendix} 
\appendix

\section{Related Work}
In this appendix, we first review related work according to two aspects: local convergence and last-iterate convergence rates.

\paragraph{Local convergence.}
Although stochastic extragradient (S-EG) and stochastic optimistic gradient descent-ascent (S-OGDA) may diverge globally even for simple bilinear problems~\citep{ryu2019ode,chavdarova2019reducing,hsieh2020explore}, local convergence guarantees have been established under suitable conditions~\citep{hsieh2019convergence,azizian2021last}. These results typically rely on stability properties in a neighborhood of a solution rather than global monotonicity.

\paragraph{Last-iterate convergence rates.}
A large body of work studies last-iterate convergence rates in the stochastic setting. In the unconstrained case, Lee et al.~\citep{lee2021fast} analyzed the stochastic fast extragradient (S-FEG) method and showed that error accumulation may occur unless the noise variance decays at a rate of $\mO(1/T)$. Under this assumption, they obtained an $\mO(1/T)$ rate in terms of the gradient norm; however, no explicit mechanism was provided to enforce such variance decay. Subsequent work~\citep{cai2022stochastic} demonstrated that, using mini-batching, a rate of $\mO(T^{-1/4})$ can be achieved in terms of stochastic oracle complexity. Moreover, by combining variance reduction techniques (e.g., PAGE) with extrapolation schemes, the Extrapolated Stochastic Halpern-Monotone (E-Halpern) method further improves the stochastic oracle complexity to $\mO(T^{-1/3})$. However, these results rely on stronger stochastic oracle assumptions, including pointwise unbiasedness and uniformly bounded variance of the noise, as well as more restrictive sampling requirements. In particular, they assume access to multi-point oracles that allow querying multiple points using the same random sample, which is essential for constructing variance-reduced estimators. Another important limitation is that these methods are restricted to unconstrained problems and do not extend to constrained settings.

In the constrained setting, strong monotonicity and related growth conditions play a central role in obtaining fast rates. Under such assumptions, S-EG achieves last-iterate convergence, and similar results extend to S-OGDA~\citep{hsieh2019convergence,mishchenko2020revisiting,kannan2019optimal,gorbunov2022stochastic,beznosikov2025distributed}. More generally, linear growth or quasi-strong monotonicity conditions yield $\mO(1/T)$ rates~\citep{vankov2023last,choudhury2023single}.
 
\paragraph{Concurrent work.}
Several recent concurrent works are closely related to ours, including
Sohrabi et al.~\cite{sohrabi2026accelerated} and
Ito et al.~\cite{itolast}.

The generalized optimistic methods with anchoring (GOMA) proposed in~\cite{sohrabi2026accelerated} establish an $\mathcal{O}(T^{-1/4})$ anytime last-iterate convergence rate in terms of the gradient norm for unconstrained convex--concave minimax problems. By contrast, our perturbation framework naturally extends to constrained convex--concave problems, where convergence is established in terms of the restricted primal--dual gap.

Another independent concurrent work~\cite{itolast} considers the same stochastic constrained convex--concave setting as ours and establishes last-iterate convergence guarantees under the assumption that the feasible domain is compact. In contrast, our analysis applies to arbitrary closed convex feasible sets, including both compact and noncompact domains, and therefore also covers the unconstrained setting as a special case. Furthermore, even on compact feasible domains, our analysis removes the logarithmic factor appearing in~\cite{itolast} while attaining the same polynomial convergence rate.



\section{Useful Lemmas}
We first collect several auxiliary results that will be used throughout the analysis. In particular, the projection inequality forms the basis for deriving the key recursive inequalities. We then present several variants of Chung's lemma that allow us to analyze different parameter schedules and derive the corresponding convergence rates in expectation.

\begin{lemma}[{\cite[Lemma A.2]{hsieh2019convergence}}]
\label{lem:proj_key}
Let $\x,\y_1,\y_2\in \R^{n+d}$ and $\U\subseteq \R^{n+d}$ be a closed convex set. 
Define $\x_1^+\coloneqq \proj_{\U} (\x-\y_1)$ and $\x_2^+\coloneqq \proj_{\U} (\x-\y_2)$. 
Then for all $\p\in \U$, we have
\begin{align*}
\|\x_2^+-\p\|^2
\leq \|\x-\p\|^2
-2\langle \y_2, \x_1^+-\p\rangle
+\|\y_2-\y_1\|^2
-\|\x_1^+-\x\|^2.
\end{align*}
\end{lemma}
This inequality will be crucial for deriving the key recursion, enabling a unified treatment of both constrained and unconstrained settings. We next introduce several variants of Chung's lemma, which provide a convenient tool for translating recursive inequalities into explicit convergence rates.
 
\begin{lemma}
	\label{lem:chung_mod}
	Let $\{a_t\}_{t\ge 0}$ be a sequence of nonnegative real numbers, and let
	$b\in\mathbb{N}$ with $b\ge 1$ be such that, for all $T\in\mathbb{N}$ and
	$t=0,\ldots,T-1$,
	\begin{equation}
		\label{eq:chung-recursion}
		a_{t+1}\leq 
		\left(1 - \frac{T^{-\delta}}{(t+b)^{\nu}}\right) a_t +
		\frac{q}{(t+b)^{2\nu}},
	\end{equation}
	where $0\leq \delta<\frac{1}{2}$, $\frac12<\nu<1-\delta$, and $q>0$. Then there exist
	constants $c,C>0$, independent of $T$, such that for all $T\ge 1+b$,
	\[
	a_T
	\le
	a_0\exp \left(-cT^{1-\nu-\delta}\right)
	+
	C qT^{\delta-\nu}.
	\] 
\end{lemma}

\begin{proof}
	Let $\lambda_t\coloneqq\frac{T^{-\delta}}{(t+b)^{\nu}}$. We have
	$0<\lambda_t\leq 1$. Iterating the recursion \eqref{eq:chung-recursion}
	from $t=0$ to $t=T-1$ yields
	\[
	a_T
	\le
	a_0\prod_{s=0}^{T-1}(1-\lambda_s)
	+
	\sum_{k=0}^{T-1}
	\frac{q}{(k+b)^{2\nu}}
	\prod_{s=k+1}^{T-1}(1-\lambda_s),
	\]
	where an empty product is understood as $1$.
	
	Using $1-x\le e^{-x}$ for all $x\ge 0$, we obtain
	\[
	\prod_{s=r}^{T-1}(1-\lambda_s)
	\le
	\exp\left\{-\sum_{s=r}^{T-1}\lambda_s\right\},
	\qquad 0\le r\le T-1.
	\]
	In particular,
	\[
	\prod_{s=0}^{T-1}(1-\lambda_s)
	\le
	\exp\left\{
	-T^{-\delta}\sum_{s=0}^{T-1}\frac{1}{(s+b)^\nu}
	\right\}.
	\]
	Since $0<\nu<1$, when $T\ge 1+b$, we have
	\[
	\sum_{s=0}^{T-1}\frac{1}{(s+b)^\nu}
	\ge
	\sum_{\lceil T/2\rceil\le s\le T-1}\frac{1}{(s+b)^\nu}
	\ge
	c_1T^{1-\nu}
	\]
	for some constant $c_1>0$ independent of $T$. Therefore, we bound the first
	term as follows:
	\[
	a_0\prod_{s=0}^{T-1}(1-\lambda_s)
	\le
	a_0\exp\left(-c_1T^{1-\nu-\delta}\right).
	\]
	
	It remains to bound the accumulated error term
	\[
	R_T
	\coloneqq
	\sum_{k=0}^{T-1}
	\frac{q}{(k+b)^{2\nu}}
	\prod_{s=k+1}^{T-1}(1-\lambda_s).
	\]
	Again using $1-x\le e^{-x}$, we have
	\[
	R_T
	\le
	 \sum_{k=0}^{T-1}
	\frac{q}{(k+b)^{2\nu}}
	\exp\left\{
	-T^{-\delta}
	\sum_{s=k+1}^{T-1}\frac{1}{(s+b)^\nu}
	\right\}.
	\]
	We split the sum into two parts:
	\[
	R_T= R_{T,1}+R_{T,2},
	\]
	where
	\[
	R_{T,1}
	\coloneqq
	 \sum_{0\le k< T/2}
	\frac{q}{(k+b)^{2\nu}}
	\exp\left\{
	-T^{-\delta}
	\sum_{s=k+1}^{T-1}\frac{1}{(s+b)^\nu}
	\right\},
	\]
	and
	\[
	R_{T,2}
	\coloneqq
	 \sum_{T/2\le k\le T-1}
	\frac{q}{(k+b)^{2\nu}}
	\exp\left\{
	-T^{-\delta}
	\sum_{s=k+1}^{T-1}\frac{1}{(s+b)^\nu}
	\right\}.
	\]
	
	First consider $R_{T,1}$. If $k<T/2$, then we have
	\[
	\sum_{s=k+1}^{T-1} \frac{1}{(s+b)^\nu}
	\ge
	\sum_{\lceil T/2\rceil\le s\le T-1}\frac{1}{(s+b)^\nu}
	\ge
	c_2T^{1-\nu}
	\]
    for some $c_2>0$.
	Hence
	\[
	R_{T,1}
	\le
	q \exp\left(-c_2T^{1-\nu-\delta}\right)
	\sum_{0\le k< T/2}
	\frac{1}{(k+b)^{2\nu}}.
	\]
	Noting that $2\nu>1$, we have
\[
\sum_{0\le k< T/2}\frac{1}{(k+b)^{2\nu}}
\le
\sum_{k=0}^{\infty}\frac{1}{(k+b)^{2\nu}}
=: C_2,
\]
where $C_2<\infty$ is a constant independent of \(T\).
Moreover, since $\frac12<\nu<1-\delta$, we have $\delta<\nu$.
Therefore,
\[
R_{T,1}
\le
C_2q\exp\left(-c_2T^{\,1-\nu-\delta}\right).
\]

	Next consider $R_{T,2}$. If $k\ge T/2$, then there exists a constant $C_3>0$ such that
	\[
	\frac{1}{(k+b)^{2\nu}}
	\le
	C_3 T^{-2\nu}.
	\]
	Moreover, for $k\le T-1$, we have
	\[
	\sum_{s=k+1}^{T-1}\frac{1}{(s+b)^\nu}
	\ge
	(T-1-k)(T+b)^{-\nu}
	\ge
	c_3(T-1-k)T^{-\nu} 
	\]
    for some $c_3>0$. Here we used $T\ge 1+b$. Therefore,
	\begin{align*}
		R_{T,2}
		\le\ &
		C_3 q T^{-2\nu}
		\sum_{T/2\le k\le T-1}
		\exp\left\{
		-c_3(T-1-k)T^{-\nu-\delta}
		\right\} \\
		\le\ &
		C_3 q T^{-2\nu}
		\sum_{j=0}^{\infty}
		\exp\left(-c_3jT^{-\nu-\delta}\right),
	\end{align*}
	where the second inequality is derived by performing a change of variables
	with $j=T-1-k$. The series on the right-hand side is geometric. Hence
	\[
	\sum_{j=0}^{\infty}
	\exp\left(-c_3jT^{-\nu-\delta}\right)
	=
	\frac{1}{1-\exp(-c_3T^{-\nu-\delta})}.
	\]
	For all $T\ge 1$, we find
	\[
	1-\exp(-c_3T^{-\nu-\delta})
	\ge
	\frac{c_3}{1+c_3}T^{-\nu-\delta}.
	\]
	Therefore, we obtain that
	\[
	R_{T,2}
	\le
	\frac{C_3(1+c_3)q}{c_3} T^{-2\nu}T^{\nu+\delta}
	=
	\frac{C_3(1+c_3)q}{c_3}  T^{\delta-\nu}.
	\]
	
	Combining the bounds for $R_{T,1}$ and $R_{T,2}$, we obtain that
	\[
	R_T
	\le
	C_2q\exp \left(-c_2T^{\,1-\nu-\delta}\right)+\frac{C_3(1+c_3)q}{c_3}  T^{\delta-\nu}.
	\]
	Therefore, there exist constants $c,C>0$, independent of $T$, such that the
	following holds for all $T\ge 1+b$:
	\[
	a_T
	\le
	a_0\exp\left(-cT^{1-\nu-\delta}\right)
	+
	Cq T^{\delta-\nu}.
	\]
	This proves the claimed bound.
\end{proof}
 
\begin{lemma}
	\label{lem:harmonic_chung}
	Let $\{a_t\}_{t\ge 0}$ be a sequence of nonnegative real numbers.
	Suppose that, for some $\tau>0$, $q >0$, and an integer $b\ge 3$, it holds for all
	$t=0,\ldots,T-1$ that
	\begin{equation}
		\label{eq:harmonic-recursion}
		a_{t+1}
		\le
		\left(1-\frac{2}{t+b}\right)a_t
		+
		\frac{q}{\tau^2(t+b)^2}.
	\end{equation}
	Then, for all $T\ge 1$,
	\[
	a_T
	\le
	\frac{a_0(b-2)(b-1)}{(T+b-2)(T+b-1)} 
	+
	\frac{qT}{\tau^2(T+b-2)(T+b-1)}.
	\]
	Consequently, there exists a universal constant $C>0$ such that
	\[
	a_T
	\le
	 \frac{ a_0C b^2}{(T+b)^2}
	+
	\frac{Cq T}{\tau^2(T+b)^2}.
	\]
\end{lemma}

\begin{proof}
	Iterating~\eqref{eq:harmonic-recursion} gives
	\[
	a_T
	\le
	a_0\prod_{s=0}^{T-1}\left(1-\frac{2}{s+b}\right)
	+
	\frac{q}{\tau^2}
	\sum_{k=0}^{T-1}
	\frac{1}{(k+b)^2}
	\prod_{s=k+1}^{T-1}
	\left(1-\frac{2}{s+b}\right),
	\]
	where an empty product is understood as $1$. Since
	\[
	1-\frac{2}{s+b}=\frac{s+b-2}{s+b},
	\]
	we have the telescoping identity
	\[
	\prod_{s=0}^{T-1}
	\left(1-\frac{2}{s+b}\right)
	=
	\frac{(b-2)(b-1)}{(T+b-2)(T+b-1)}.
	\]
	Similarly, for every $k=0,\ldots,T-1$,
	\[
	\prod_{s=k+1}^{T-1}
	\left(1-\frac{2}{s+b}\right)
	=
	\frac{(k+b-1)(k+b)}{(T+b-2)(T+b-1)}.
	\]
	Therefore, the accumulated error term is bounded by
	\begin{align*}
		&
		\frac{q}{\tau^2}
		\sum_{k=0}^{T-1}
		\frac{1}{(k+b)^2}
		\frac{(k+b-1)(k+b)}{(T+b-2)(T+b-1)}
		\\
		\leq\ &
		\frac{q}{\tau^2}
		\frac{1}{(T+b-2)(T+b-1)}
		\sum_{k=0}^{T-1}1
		=
		\frac{qT}{\tau^2(T+b-2)(T+b-1)}.
	\end{align*}
	Combining the two bounds proves the first claim. The second claim follows from
	$b\ge 3$ and elementary comparisons between $T+b-2$ and $T+b$.
\end{proof}
 \begin{lemma}[{\cite[Lemma~1]{chung1954stochastic}}]
\label{lem:chung_harmonic_general}
Let $\{a_t\}_{t\ge0}$ be a sequence of nonnegative real numbers. Suppose that, for some
$c>0$, $r>0$, $C>0$, and $b\in\mathcal N_+$, it holds that
\[
a_{t+1}
\le
\left(1-\frac{c}{t+b}\right)a_t
+
\frac{C}{(t+b)^{r+1}},
\qquad \forall t\ge0,
\]
where $c>r$. Then,
\[
a_t\leq\mO(t^{-r}).
\] 
\end{lemma}
 \begin{lemma}[{\cite[Lemmas~3,4]{chung1954stochastic}}]
\label{lem:chung}
Let $\{a_t\}_{t\in\mathbb{N}}$ be a sequence of real numbers, and let $b\in\mathbb{N}_+$ be such that, for all $t\in\mathbb{N}$,
\begin{equation*} 
a_{t+1}\leq 
\left(1 - \frac{q}{(t+b)^{\nu}}\right) a_t +
\frac{q'}{(t+b)^{r+\nu}},
\end{equation*}
where $0 < \nu < 1$ and $r,q,q' > 0$.  
Then,
\[
a_t=\mO(t^{-r}) .
\]
\end{lemma}

\begin{lemma}\label{lem:chung_new}
Let $\{a_t\}_{t\in\mathbb{N}}$ be a sequence of real positive numbers, and let $b\in\mathbb{N}_+$. Suppose the following bounds holds for all $t\in\mathbb{N}$, 
\begin{equation*} 
a_{t+1}\leq 
\left(1 - \frac{c}{(t+b)^{p}}\right) a_t +
\frac{C_1}{(t+b)^{q_1}}+\frac{C_2}{(t+b)^{q_2}},
\end{equation*}
where $0 < p < 1$, $q_1,q_2>p$, and $c,C_1, C_2 > 0$.  
Then,
\[
a_t\leq \mO(t^{-\left(\min\{q_1,q_2\}-p\right)}).
\]
\end{lemma}

\begin{proof}
Let
\[
r_1\coloneqq q_1-p,
\qquad
r_2\coloneqq q_2-p.
\]
Since $q_1,q_2>p$, we have $r_1,r_2>0$.

Fix $i\in\{1,2\}$. Define the auxiliary sequence
$\{u_t^{(i)}\}_{t\ge 1}$ by
\[
u^{(i)}_{1}=a_{1},
\qquad
u^{(i)}_{t+1}
=
\left(1-\frac{c}{(t+b)^p}\right)u^{(i)}_{t}
+
\frac{C_i}{(t+b)^{r_i+p}}.
\]
By induction on $t$, the recursion satisfied by $\{a_t\}$ implies
\[
a_t \le u_t^{(1)}+u_t^{(2)},
\qquad \forall t\ge 1.
\]
Applying Lemma~\ref{lem:chung} with
\[
\nu=p,\qquad q=c,\qquad q'=C_i,\qquad r=r_i,
\]
yields
\[
\limsup_{t\to\infty} t^{r_i}u_t^{(i)}
\le
\frac{C_i}{c},
\qquad i=1,2.
\]
Hence, for every $\varepsilon>0$, there exists $T_\varepsilon$ such that
\[
u_t^{(i)}
\le
\frac{C_i/c+\varepsilon}{t^{r_i}},
\qquad
\forall t\ge T_\varepsilon,
\]
for $i=1,2$.

Combining the above estimates with
$a_t \le u_t^{(1)}+u_t^{(2)}$, we obtain
\[
a_t
\le
\frac{C_1/c+\varepsilon}{t^{q_1-p}}
+
\frac{C_2/c+\varepsilon}{t^{q_2-p}},
\qquad
\forall t\ge T_\varepsilon.
\]
Since $\varepsilon>0$ is arbitrary, it follows that
\[
a_t
=
\mathcal{O} \left(t^{-(q_1-p)}\right)
+
\mathcal{O} \left(t^{-(q_2-p)}\right).
\]

Equivalently,
\[
a_t
=
\mathcal{O} \left(t^{-(\min\{q_1,q_2\}-p)}\right).
\]
This completes the proof.
\end{proof}

\section{Proof of Lemma~\ref{lem:per_sol}}

\begin{proof}[Proof of Lemma~\ref{lem:per_sol}] (i) By the optimality condition of the original problem, we have
\[
    \langle V(\u^\star),\u_\tau^\star-\u^\star\rangle \ge 0.
\]
Similarly, the optimality condition of the perturbed problem gives
\[
    \langle V(\u_\tau^\star)+\tau \u_\tau^\star,
    \u^\star-\u_\tau^\star\rangle \ge 0.
\]
Adding the two inequalities yields
\[
    \tau \langle \u_\tau^\star,\u_\tau^\star-\u^\star\rangle
    \le
    \langle V(\u^\star)-V(\u_\tau^\star),\u_\tau^\star-\u^\star\rangle.
\]
By the monotonicity of $V$ (Assumption~\ref{ass:cc}), the right-hand side is non-positive, hence
\[
    \langle \u_\tau^\star,\u_\tau^\star-\u^\star\rangle \le 0.
\]
Expanding this gives
\[
    \|\u_\tau^\star\|^2 \le \langle \u_\tau^\star,\u^\star\rangle.
\]
Applying Cauchy--Schwarz yields
\[
    \|\u_\tau^\star\|^2 \le \|\u_\tau^\star\|\,\|\u^\star\|,
\]
which implies $\|\u_\tau^\star\|\le \|\u^\star\|$ if $\u_\tau^\star\neq 0$, and the result is trivial otherwise.
	(ii)  
By the optimality condition of $\u_\tau^\star$, we have 
\[
\langle V(\u_\tau^\star)+\tau \u_\tau^\star, \u_\tau^\star-\u_{\tau'}^\star\rangle \le 0.
\] 
Similarly, by the optimality condition of $\u_{\tau'}^\star$, we have
\[
\langle V(\u_{\tau'}^\star)+\tau' \u_{\tau'}^\star, \u_{\tau}^\star-\u_{\tau'}^\star\rangle \ge 0.
\] 
Combining these two inequalities gives
\begin{align*}
    &\langle V(\u_{\tau'}^\star)+\tau' \u_{\tau'}^\star-\left(  V(\u_\tau^\star)+\tau \u_\tau^\star\right), \u_{\tau}^\star-\u_{\tau'}^\star\rangle \ge 0\\
    \Leftrightarrow\ & \langle \tau' \u_{\tau'}^\star  -\tau \u_\tau^\star , \u_{\tau}^\star-\u_{\tau'}^\star\rangle \ge \langle V(\u_{\tau'}^\star)-V(\u_\tau^\star) ,\u_{\tau'}^\star-\u_{\tau}^\star\rangle.
\end{align*}
Using Assumption~\ref{ass:cc}, we get the right-hand side is non-negative. Therefore, we have 
\begin{align*}
    &\langle \tau' \u_{\tau'}^\star  -\tau \u_\tau^\star , \u_{\tau}^\star-\u_{\tau'}^\star\rangle\geq 0 \\
    \Leftrightarrow\ &  \langle  -\left(\tau-\tau'\right)\u_{\tau'}^\star-\tau \left(\u_{\tau}^\star-\u_{\tau'}^\star\right)    , \u_{\tau}^\star-\u_{\tau'}^\star\rangle\geq 0\\
     \Leftrightarrow\ & -\left(\tau-\tau'\right) \langle \u_{\tau'}^\star, \u_{\tau}^\star-\u_{\tau'}^\star\rangle\geq \tau \|\u_{\tau}^\star-\u_{\tau'}^\star\|^2.
\end{align*} 
By Cauchy-Schwarz inequality, we further get 
\[
\tau  \|\u_{\tau}^\star-\u_{\tau'}^\star\|^2 \leq |\tau-\tau'| \|\u_{\tau'}^\star\|\| \u_{\tau}^\star-\u_{\tau'}^\star\|.
\]
If $\|\u_{\tau}^\star-\u_{\tau'}^\star\|=0$, the desired inequality is trivial. Otherwise, dividing both sides by $\tau  \|\u_{\tau}^\star-\u_{\tau'}^\star\| $ yields
\[
 \|\u_{\tau}^\star-\u_{\tau'}^\star\|\le
\frac{|\tau-\tau'|}{\tau}\|\u_{\tau'}^\star\|.
\] 
This finishes the proof. 
\end{proof}

\section{Proof of Lemma~\ref{lem:bd_iter}} 
\begin{proof}[Proof of Lemma~\ref{lem:bd_iter}]
We present a unified proof for both the horizon-dependent and the anytime settings, as well as for PS-EG and PS-OGDA.
For convenience, define
\[
\widehat{\u}_t^\star=
\begin{cases}
\u_\tau^\star, & \text{for the horizon-dependent setting},\\[1mm]
\u_{\tau_t}^\star, & \text{for the anytime setting}.
\end{cases}
\]
We first observe that all parameter choices considered in
Theorems~\ref{thm:rate_L2} and~\ref{thm:rate_any} satisfy $\sum_{t=0}^{\infty}\eta_t^2<\infty.$ 
Moreover, in the anytime setting,
\[
\sum_{t=0}^{\infty}
\frac{(\tau_t-\tau_{t+1})^2}
{\eta_t\tau_t^3}
<\infty,
\qquad
\sum_{t=0}^{\infty}
(\tau_t-\tau_{t+1})^2\eta_t^2
<\infty.
\]

We next establish the boundedness with respect to the perturbed saddle point.

For PS-EG, from equations~\eqref{eq:eg_1} and~\eqref{eq:eg_recur}, we have
\[
\E[\|\u_{t+1}-\widehat{\u}_{\tau_{t+1}}^\star\|^2]
\le
\E[\|\u_t-\widehat{\u}_{\tau_t}^\star\|^2]
+r_t,
\]
where
\[
r_t=
\begin{cases}
6\eta_t^2\sigma^2,
& \text{(horizon-dependent)},\\[1mm]
12\eta_t^2\sigma^2
+
\dfrac{5(\tau_t-\tau_{t+1})^2}
{\eta_t\tau_t^3}\|\u^\star\|^2,
& \text{(anytime)}.
\end{cases}
\]

For PS-OGDA, the equations~\eqref{eq:og_1} and~\eqref{eq:og_recur} imply
\[
\Psi_{t+1}
\le
\Psi_t+\widetilde r_t,
\]
where
\[
\Psi_t\coloneqq
 \E[ \|\u_{t}-\u_{\tau_t}^\star\|^2]+ 3\eta_{t-1}^2 \E [\| W(\u_{t-\frac{1}{2}})-W(\u_{t-\frac{3}{2}})  \|^2]+\frac{12\left(\tau_{t}-\tau_{t-1}\right)^2\eta_{t}^2}{1-\frac{\eta_t\tau_t}{2}} \E[\|\u_{t-\frac{1}{2}}\|^2]
\]
and
\[
\widetilde r_t=
\begin{cases}
\left(144\eta_{t-1}^2\left(L+\tau\right)^2+6\right)\eta_t^2\sigma^2,
& \text{(horizon-dependent)},\\[2mm]
16\eta_t^2\sigma^2
+
\dfrac{5(\tau_t-\tau_{t+1})^2}
{\eta_t\tau_t^3}\|\u^\star\|^2
+
96(\tau_{t+1}-\tau_t)^2\eta_{t+1}^2
\|\u^\star\|^2,
& \text{(anytime)}.
\end{cases}
\]

Since all error sequences are summable, there exists a constant $C>0$ such that $\sum_t r_t<  C $ and $\sum_t \widetilde r_t<C$.  By applying a discrete Gr\"onwall argument, we obtain  $\sup_{t\geq 0}  \E [\|\u_{t}-\u_{\tau_t}^\star\|^2]\leq\E[\|\u_0-\u_{\tau_0}^\star\|^2]+C$. Thus
   \begin{align*}
       \sup_{t\geq 0}  \E [\|\u_{t}-\u^\star\|^2] \leq\ & \sup_{t\geq 0}  2\E [\|\u_{t}-\u_{\tau_t}^\star\|^2]+2\|\u_{\tau_t}^\star-\u^\star\|^2\\
       \leq\ & 2\E[\|\u_0-\u_{\tau_0}^\star\|^2]+2C+2\|\u_{\tau_t}^\star-\u^\star\|^2 .
   \end{align*}
By applying Young's inequality and the triangular inequality, we can bound the terms involving $\u_{\tau_t}^\star$ as follows:
\begin{align*}
2\mathbb{E}[\|\u_0-\u_{\tau_0}^\star\|^2] + 2\|\u_{\tau_t}^\star-\u^\star\|^2 
&\leq 4\mathbb{E}[\|\u_0-\u^\star\|^2] + 4\|\u_{\tau_0}^\star-\u^\star\|^2+2\|\u_{\tau_t}^\star-\u^\star\|^2  \\
&\leq 4\mathbb{E}[\|\u_0-\u^\star\|^2] + 8\|\u_{\tau_0}^\star\|^2+4\|\u_{\tau_t}^\star\|^2 + 12\|\u^\star\|^2.
\end{align*}
By Lemma~\ref{lem:per_sol}, we have $\|\u_{\tau_t}^\star\| \leq \|\u^\star\|$ for all $t\geq 0$ and noting that $\|\u^\star\|^2 \leq 2\|\u_0-\u^\star\|^2 + 2\|\u_0\|^2$, the total expression is bounded by:
\begin{align*}
 \sup_{t\geq 0}  \E [\|\u_{t}-\u^\star\|^2] \leq\ &2\E[\|\u_0-\u_{\tau_0}^\star\|^2]+2C+2\|\u_{\tau_t}^\star-\u^\star\|^2 \\
\leq\ & 52\E[\|\u_0-\u^\star\|^2]+48\|\u_0\|^2+2C.
\end{align*}
This completes the proof.
\end{proof}

\end{document}